\documentclass
[11pt]{amsart}

\usepackage[USenglish]{babel}
\usepackage[T1]{fontenc} 
\usepackage[utf8,latin1]{inputenc}

\usepackage{amsmath,amsthm,amssymb,amsfonts,mathrsfs}
\usepackage{dsfont}
\usepackage{mathtools}

\DeclareMathAlphabet{\mathbbold}{U}{bbold}{m}{n} 

\usepackage{enumitem}
\usepackage{nicematrix}
\usepackage{lmodern}
\usepackage[all]{xy}

\usepackage[bookmarks=true]{hyperref}
\usepackage{xcolor}
\hypersetup{
    colorlinks,
    linkcolor={red!50!black},
    citecolor={blue!50!black},
    urlcolor={blue!80!black},
}

\usepackage[bottom=3cm, top=3cm, left=3.5cm, right=3.5cm]{geometry}

\mathtoolsset{showonlyrefs}  

\newcommand{\R}{\mathbb{R}}
\newcommand{\Z}{\mathcal{Z}}
\newcommand{\A}{\mathcal{A}}

\newcommand{\D}{\mathcal{D}}

\newcommand{\sr}{sub-Riemannian }
\newcommand{\ariem}{almost-Riemannian }

\newcommand{\surf}{\Sigma}
\newcommand{\srexp}{\mathrm{exp}}

\newcommand{\zero}{\mathbf{0}}
\newcommand{\eps}{\varepsilon}
\newcommand{\di}{\mathsf d}
\newcommand{\de}{{\mathrm d}}
\newcommand{\mres}{\mathbin{\vrule height 1.6ex depth 0pt width
0.13ex\vrule height 0.13ex depth 0pt width 1.3ex}}
\newcommand{\ord}{\mathrm{ord}}
\newcommand{\rest}[2]{{{\rm restr}_{#1}^{#2}}}
\DeclareMathOperator{\Geo}{Geo}
\newcommand{\p}{\mathtt p}
\newcommand{\Prob}{\mathscr{P}}
\newcommand{\Leb}{\mathscr{L}}

\newcommand{\cd}{{\sf CD}}
\newcommand{\m}{\mathfrak{m}}
\newcommand{\q}{\mathfrak{q}}
\newcommand{\f}{\mathfrak{f}}
\newcommand{\h}{\mathfrak{h}}
\newcommand{\Q}{\mathfrak{Q}}

\DeclareMathOperator{\spn}{span}

\theoremstyle{plain}
\newtheorem{thm}{Theorem}[section]

\newtheorem{lem}[thm]{Lemma}
\newtheorem*{lem*}{Lemma}
\newtheorem{prop}[thm]{Proposition}

\theoremstyle{definition}
\newtheorem{defn}[thm]{Definition}
\newtheorem{example}[thm]{Example}

\theoremstyle{remark}
\newtheorem{rmk}[thm]{Remark}

\author[M. Magnabosco]{Mattia Magnabosco}
\author[T. Rossi]{Tommaso Rossi}
\address{Institut f\"ur Angewandte Mathematik, Universit\"at Bonn, Bonn, Germany}
\email{\href{mailto:magnabosco@iam.uni-bonn.de}{magnabosco@iam.uni-bonn.de}}
\email{\href{mailto:rossi@iam.uni-bonn.de}{rossi@iam.uni-bonn.de}}

\title[Almost-Riemannian manifolds do not satisfy the \texorpdfstring{$\cd$}{CD} condition]{Almost-Riemannian manifolds do not satisfy the curvature-dimension condition}

\begin{document}

\begin{abstract}
The Lott-Sturm-Villani curvature-dimension condition $\cd(K,N)$ provides a synthetic notion for a metric measure space to have curvature bounded from below by $K$ and dimension bounded from above by $N$. It was proved by Juillet in \cite{MR4201410} that a large class of \sr manifolds do not satisfy the $\cd(K,N)$ condition, for any $K\in\R$ and $N\in(1,\infty)$. However, his result does not cover the case of almost-Riemannian manifolds. In this paper, we address the problem of disproving the $\cd$ condition in this setting, providing a new strategy which allows us to contradict the one-dimensional version of the $\cd$ condition. In particular, we prove that $2$-dimensional almost-Riemannian manifolds and strongly regular almost-Riemannian manifolds do not satisfy the $\cd(K,N)$ condition for any $K\in\R$ and $N\in(1,\infty)$.\vspace{4pt} \vspace{4pt}  \vspace{4 pt} \\
\textbf{Data availability statement}: All data generated or analysed during this study are included in this published article.
\end{abstract}

\maketitle 

\tableofcontents

\section{Introduction}
In their seminal works Lott-Villani \cite{MR2480619} and Sturm \cite{MR2237206,MR2237207} introduced a synthetic notion of curvature-dimension bounds, which is heavily based on the theory of Optimal Transport. They noticed that, in a Riemannian manifold, a uniform lower bound on the Ricci curvature, together with an upper bound on the dimension, is equivalent to a convexity property of the R\'enyi entropy functionals in the Wasserstein space. This allowed them to define a consistent notion of curvature-dimension bounds for metric measure spaces, known as $\cd$ condition. While in the Riemannian setting, the $\cd$ condition is equivalent to having bounded geometry, an analogue result does not hold in the \sr setting. Sub-Riemannian geometry is a far-reaching generalization of Riemannian geometry: given a smooth manifold $M$, we define a smoothly varying scalar product only on a subset of \emph{horizontal} directions $\D_p\subset T_pM$ (called distribution) at each point $p\in M$. Under the so-called H\"ormander condition, $M$ is horizontally-path connected, and the usual length-minimization procedure yields a well-defined distance $\di$. In particular, differently from what happens in Riemannian geometry, the rank of the distribution $r(p)=\dim \D_p$ may be strictly less than the dimension of the manifold and may vary with the point. In general, we can not expect the $\cd$ condition to hold for \emph{truly} \sr manifolds. This statement is confirmed by the following result by Juillet.

\begin{thm}[{\cite[Cor.\ 1.2]{MR4201410}}]
\label{thm:juillet}
Let $M$ be a complete \sr manifold with $ \dim M \geq 3$, equip\-ped with a smooth positive (i.e.\ with strictly positive density) measure $\m$. Assume that the possibly varying rank of the distribution is smaller than $\dim M-1$. Then, $(M,\di,\m)$ does not satisfy the $\cd(K,N)$ condition for any $K\in\R$ and $N\in(1,\infty)$.
\end{thm}

\noindent While this result is quite general, it does not include many cases of interest, such as \emph{almost-Riemannian geometry}. Roughly speaking, an \ariem manifold is a \sr manifold where the rank of the distribution coincides with the dimension of $M$, at almost every point\footnote{But not at every point, otherwise the structure would be Riemannian.}. For this reason, the technique used to prove Theorem \ref{thm:juillet} can not be adapted to this setting. Indeed, it relies on the construction of two Borel subsets for which the Brunn-Minkowski inequality does not hold, namely, for all $R,\eps>0$, one can find $A,B\subset M$ such that $\mathrm{diam}(A\cup B)<R$, $\m(A)\approx\m(B)$, and such that there exists $t\in (0,1)$ for which
\begin{equation}
\label{eq:bm_ineq}
    \m(Z_t(A,B))\leq \frac{1}{2^{\mathcal{N}-\dim M}}\m(B)(1+\eps),
\end{equation}
where $Z_t(A,B)$ denotes the $t$-intermediate set between $A$ and $B$ and $\mathcal{N}$ is the so-called \emph{geodesic dimension} of $M$, see  \cite{MR3502622} or \cite[Def.\ 5.47]{MR3852258} for a precise definition. 
The inequality \eqref{eq:bm_ineq} allows to contradict the Brunn-Minkowski inequality if and only if the geodesic dimension $\mathcal N$ is strictly greater than $\dim M$. However, in the \ariem setting, $\mathcal N=\dim M$ almost everywhere, making this construction inconclusive. We mention that Juillet in \cite{MR2648260} disproved the $\cd$ condition in the simple example of the standard Grushin plane (cf. Example \ref{ex:grushin}) equipped with the Lebesgue measure, by direct computations. Heuristically, disproving the $\cd$ condition in \ariem manifolds is a more challenging task, since they behave in some sense like non-complete Riemannian manifolds. Thus, a new strategy is needed. 

\medskip 

Our idea is to exploit the one-dimensional characterization of the $\cd$ condition: 
\begin{equation}\label{eq:camon}
    \cd(K,N)\quad\Rightarrow\quad \cd^1(K,N),
\end{equation}
proven by Cavalletti and Mondino in \cite{MR3648975}, and contradict the $\cd^1(K,N)$ condition. For any $1$-Lipschitz function $u$, the latter relies on a disintegration of the reference measure, associated with $u$, in one-dimensional transport rays and requires the $\cd(K,N)$ condition to hold along them. The main advantage in dealing with one-dimensional $\cd$ spaces is related to a differential characterization of the $\cd$ densities, (cf.\ Lemma \ref{lem:diff_char}), which is easier to disprove compared with the convexity of the R\'enyi entropy.  In Section \ref{subsection:onedimCD}, we present a local version of the one-dimensional characterization \eqref{eq:camon} (cf. Proposition \ref{prop:localization}), which permits to exploit the local structure of \sr manifolds. Then, in the case of an \ariem manifold, equipped with a smooth positive measure $\m$, we are able to explicitly compute the disintegration and verify that the \emph{one-dimensional} $\cd(K,N)$ condition along the rays does not hold for any $K\in\R$ and $N\in (1,\infty)$. Our main result is the following, cf.\ Theorems \ref{thm:2dim_cd} and \ref{thm:strongly_reg_cd}. We refer to Sections \ref{sec:prel} and \ref{sec:strongly_reg} for precise definitions. 

\begin{thm}
\label{thm:intro}
Let $M$ be a complete \ariem manifold and let $\m$ be any smooth positive (i.e.\ with strictly positive density) measure on $M$. Assume $M$ is either of dimension $2$ or strongly regular. Then, the metric measure space $(M,\di,\m)$ does not satisfy the $\cd(K,N)$ condition for any $K\in\R$ and $N\in (1,+\infty)$. 
\end{thm}

\noindent Remarkably, for $2$-dimensional \ariem manifolds, we do not require any additional assumption on the structure of the singular region $\Z$, see \eqref{eq:singular} for the precise definition. However, as soon as the dimension of the manifold increases, the complexity of the computations prevents us to treat the general case and we need an auxiliary control on the behavior of the distribution. Nonetheless, we stress that our procedure is algorithmic and can be applied to any explicit example of \ariem manifold. This algorithmic procedure has been implemented in the software \emph{Mathematica}, see \cite{script}.

\medskip

A crucial tool for proving Theorem \ref{thm:intro} will be a truly \sr phenomenon, namely the existence of \emph{characteristic points}. For an embedded hypersurface $\surf\subset M$, a characteristic point is a point where the distribution is tangent to $\surf$. Of course, such points do not exist in Riemannian geometry, but as soon as the rank of the distribution $r(p)<\dim M$ for some $p\in M$, they can appear. Usually, characteristic points are source of subtle technical problems, mostly related to the low regularity of the (signed) distance $\delta_\surf$ from $\surf$. Indeed, although being $1$-Lipschitz with respect to $\di$, $\delta_\surf$ is not smooth around characteristic points (and not even Lipschitz in coordinates). In the proof of Theorem \ref{thm:intro}, we choose a suitable hypersurface $\surf$, we build the disintegration of $\m$ associated with a localized version of $\delta_\surf$ and we exploit its singular behavior to contradict the differential characterization of the one-dimensional $\cd(K,N)$ condition. In particular, $\surf$ is chosen to be \emph{transverse} to the singular region of $M$ in such a way $\surf\cap\Z$ exhibits characteristic points; we can then exploit the Riemannian structure at points of $\surf\setminus\Z$ to describe the degeneration of $\delta_\surf$ in the disintegration of $\m$. For example, in the standard Grushin plane, where the singular region is $\Z=\{x=0\}$, a suitable transverse hypersurface is $\surf=\{y=0\}$.  

It is worth mentioning that there exists a weaker synthetic notion of curvature bounds, introduced by Ohta in \cite{MR2341840}, called \emph{measure contraction property} or $\mathsf{MCP}$ condition. This property seems to be more suited to \sr geometry, see for example \cite{MR3848070,MR3935035,MR4019096,MR4245620}. Finally, we refer to \cite{MR4373164} for a relaxation of the $\cd$ condition, called \emph{quasi-curvature-dimension} condition, which holds for a certain class of \sr manifolds. However, it is not known whether these weaker conditions hold for a general \ariem manifold.

\subsection*{Acknowledgments} The authors gratefully acknowledge support from the European Research Council (ERC) under the program ERC-AdG RicciBounds, grant agreement No. 694405. The authors are thankful to Fabio Cavalletti for stimulating discussion regarding the $\cd^1$ condition. We would like to thank anonymous referee for the careful reading of the manuscript and the valuable comments.

\medskip

After submitting this work, Rizzi and Stefani proved in \cite{rizzistefani} that every \sr manifold does not satisfy the $\cd(K,\infty)$ condition, using different techniques.


\section{Preliminaries}
\label{sec:prel}
\subsection{Almost-Riemannian geometry}

We recall some basic facts about almost-Rie\-mannian geometry, following \cite{ABB-srgeom}. 
\begin{defn}
Let $M$ be a smooth, connected manifold. A \emph{sub-Riemannian structure} on $M$ is a triple $(\mathbb{U},\xi,(\cdot|\cdot))$ satisfying the following conditions:
\begin{enumerate}[label=\roman*)]
    \item $\pi_{\mathbb U}\colon \mathbb U\rightarrow M$ is a Euclidean bundle of rank $k$ with base $M$, namely for all $p\in M$, the fiber $\mathbb U_p$ is a vector space equipped with a scalar product $(\cdot|\cdot)_p$, which depends smoothly on $p$;
    \item The map $\xi\colon\mathbb U\rightarrow TM$ is a morphism of vector bundles, i.e. $\xi$ is smooth and such that the following diagram commutes:
    \begin{equation}
        \xymatrix{
        \mathbb U \ar[rd]_{\pi_{\mathbb U}} \ar[r]^{\xi} & TM \ar[d]^{\pi_M}\\
                                       & M
        }
    \end{equation}
    where $\pi_M\colon TM\rightarrow M$ denotes the canonical projection of the tangent bundle. 
    \item The distribution $\D=\{\xi(\sigma)\mid \sigma\colon M\rightarrow \mathbb U \text{ smooth section}\}\subset TM$ satisfies the \emph{H\"ormander condition} (also known as bracket-generating condition), namely 
\begin{equation}
    \mathrm{Lie}_p(\D)=T_pM,\qquad\forall\,p\in M.
\end{equation}
\end{enumerate}
With a slight abuse of notation, we say that $M$ is a \sr manifold.
\end{defn}

\noindent Let $(\mathbb U,\xi,(\cdot|\cdot))$ be a \sr structure on $M$. We can define the sub-Rieman\-nian norm on $\D$ as
\begin{equation}
\label{eq:norm}
    \|v\|^2_p=\inf\{(u|u)_p\mid u\in\mathbb U_p,\ \xi(u)=v\},\qquad\forall\,v\in\D_p,\ p\in M.
\end{equation}
The norm \eqref{eq:norm} is well-defined since the infimum is actually a minimum and it induces a scalar product $g_p$ on $\D_p$ by polarization. Notice that different \sr structures on $M$ may define the same distributions and induced norms. This is the case for equivalent \sr structures. 

\begin{defn}
Let $(\mathbb{U}_1,\xi_1,(\cdot|\cdot)_1)$, $(\mathbb{U}_2,\xi_2,(\cdot|\cdot)_2)$ be two \sr structures on $M$. These are said to be \emph{equivalent} if the following conditions hold:
\begin{enumerate}[label=\roman*)]
    \item There exists a Euclidean bundle $(\mathbb V,(\cdot|\cdot)_\mathbb{V})$ and two surjective bundle morphisms $p_i\colon\mathbb V\rightarrow \mathbb{U}_i$ such that the following diagram is commutative
     \begin{equation}
        \xymatrix{                 
        \mathbb V \ar[r]^{p_1} \ar[d]_{p_2} &  \mathbb{U}_1 \ar[d]^{\xi_1}\\
         \mathbb{U}_2 \ar[r]_{\xi_2}& TM
        }
    \end{equation}
    \item The projections $p_i$'s are compatible with the scalar products defined on $\mathbb{U}_i$, namely
    \begin{equation}
        (u|u)_i=\min\{(v|v)_\mathbb{V}\mid p_i(v)=u\},\qquad\forall\,u\in\mathbb{U}_i,\quad i=1,2.
    \end{equation}
\end{enumerate}
\end{defn}

\begin{defn}
Let $M$ be a \sr manifold. The \emph{minimal bundle rank} is the infimum of the rank of Euclidean bundles inducing equivalent structures on $M$. For $p\in M$, the \emph{local minimal bundle rank} of $M$ at $p$ is the minimal bundle rank of the structure when restricted to a sufficiently small neighborhood $\mathcal{U}_p$.
\end{defn}

\begin{defn}[Almost-Riemannian structure]
\label{def:ARS}
Let $M$ be a connected, smooth manifold of dimension $n+1$ and let $(\mathbb{U},\xi,(\cdot|\cdot))$ be a \sr structure on $M$. We say that $M$ is an \emph{\ariem manifold} if the local minimal bundle rank of the structure is $n+1$.
\end{defn}

We denote by $\Z$ the set of \emph{singular points}, namely those points where the distribution has not full rank: 
\begin{equation}\label{eq:singular}
    \Z=\{p\in M\mid \dim(\D_p)<n+1\}.
\end{equation}
Notice that $\Z$ is closed, since the rank of the distribution is lower semi-continuous. We say that a point is \emph{Riemannian} if it belongs to $M\setminus\Z$.

\begin{rmk}
If the singular set is empty, then the structure on $M$ is Riemannian. Therefore, we will always tacitly assume that $\Z\neq\emptyset$.
\end{rmk}

A local orthonormal frame for the distribution is the image through $\xi$ of a local orthonormal frame for $\mathbb{U}$. Consequently, by definition of \ariem manifold of dimension $n+1$, it consists of exactly $n+1$ vector fields which are linearly independent only at Riemannian points. In particular, local orthonormal frames are standard Riemannian orthonormal frames around Riemannian points.

\begin{example}[Grushin plane]\label{ex:grushin}
Let $M=\R^2$ and consider the \sr structure given by $\mathbb{U}=\R^2\times\R^2$ with the standard Euclidean scalar product on fibers and
\begin{equation}
    \xi\colon\mathbb{U}\rightarrow T\R^2;\qquad \xi(x,z,u_1,u_2)=(x,z,u_1,xu_2).
\end{equation}
As one can check, the resulting distribution is generated by the orthonormal vector fields $X=\partial_x$, $Y=x\partial_z$. The local minimal bundle rank is equal to $2$, thus the structure is almost-Riemannian. In this case the singular region is $\Z=\{x=0\}$ and $\{X,Y\}$ is a (global) orthonormal frame. 
\end{example}

\begin{rmk}
Any truly \sr structure (meaning that is not Riemannian) of rank $2$ on a $2$-dimensional manifold is always almost-Riemannian, in the sense of Definition \ref{def:ARS}. 
\end{rmk}

\subsection{Almost-Riemannian distance}
Let $(\mathbb{U},\xi,(\cdot|\cdot))$ be an \ariem structure on $M$. We say that $\gamma : [0,T] \to M$ is a \emph{horizontal curve}, if it is absolutely continuous and
\begin{equation}
\dot\gamma(t)\in\D_{\gamma(t)}, \qquad\text{for a.e.}\,t\in [0,T].
\end{equation}
This implies that there exists a measurable function $u:[0,T]\to\mathbb U$, such that
\begin{equation}
\pi_{\mathbb U}(u(t))=\gamma(t),\qquad\dot\gamma(t)=\xi(u(t)), \qquad \text{for a.e.}\, t \in [0,T].
\end{equation} 
Moreover, we have that $u\in L^\infty([0,T],\mathbb U)$, see \cite[Lemma 3.12]{ABB-srgeom}, therefore the map $t\mapsto \|\dot\gamma(t)\|$ is integrable on $[0,T]$. We define the \emph{length} of a horizontal curve as follows:
\begin{equation}
\ell(\gamma) = \int_0^T \|\dot\gamma(t)\| \de t.
\end{equation}
The \emph{\ariem distance} on $M$ is defined, for any $p,q\in M$, by
\begin{equation}\label{eq:infimo}
\di(p,q) = \inf\{\ell(\gamma)\mid \gamma \text{ horizontal curve between $p$ and $x$} \}.
\end{equation}
By Chow-Rashevskii theorem (see for example \cite[Thm. 5.9]{AS-GeometricControl}), the bracket-genera\-ting assumption ensures that the distance $\di\colon M\times M\to\R$ is finite and continuous. Furthermore it induces the same topology as the manifold one. We say that $M$ is \emph{complete}, if the metric space $(M,\di)$ is.

\subsection{Geodesics and Hamiltonian flow}


A \emph{geodesic} is a horizontal curve $\gamma \colon[0,T] \rightarrow M$, parameterized with constant speed, such that any sufficiently short segment is length-minimizing. The \emph{almost-Riemannian Hamiltonian} is the function on the cotagent space $H\in C^\infty(T^*M)$ defined by
\begin{equation}
\label{eq:Hamiltonian}
H(\lambda)= \frac{1}{2}\sum_{i=0}^n \langle \lambda, X_i \rangle^2, \qquad \lambda \in T^*M,
\end{equation}
where $\{X_0,\ldots,X_n\}$ is a local orthonormal frame for the almost-Riemannian structure, and $\langle \lambda, \cdot \rangle $ denotes the action of covectors on vectors. The \emph{Hamiltonian vector field} $\vec H$ on $T^*M$ is defined by $\varsigma(\cdot,\vec H)=dH$, where $\varsigma\in\Lambda^2(T^*M)$ is the canonical symplectic form. Solutions $\lambda\colon[0,T] \rightarrow T^*M$ to the \emph{Hamilton equations}
\begin{equation}\label{eq:Hamiltoneqs}
\dot{\lambda}(t) = \vec{H}(\lambda(t)),
\end{equation}
are called \emph{normal extremals}. Their projections $\gamma(t) = \pi(\lambda(t))$ on $M$, where $\pi\colon T^*M\rightarrow M$ is the bundle projection, are locally length-minimizing horizontal curves parameterized with constant speed, and are called \emph{normal geodesics}. If $\gamma$ is a normal geodesic with normal extremal $\lambda$, then its speed is given by $\| \dot\gamma \|_g = \sqrt{2H(\lambda)}$. In particular
\begin{equation}
\label{eq:speed}
\ell(\gamma|_{[0,t]}) = t \sqrt{2H(\lambda(0))},\qquad \forall\, t\in[0,T].
\end{equation} 
There is another class of length-minimizing curves in sub-Riemannian geometry, called \emph{abnormal} or \emph{singular}. As for the normal case, to these curves it corresponds an extremal lift $\lambda(t)$ on $T^*M$, which however may not follow the Hamiltonian dynamics \eqref{eq:Hamiltoneqs}. 
Here we only observe that an abnormal extremal lift $\lambda(t)\in T^*M$ satisfies
\begin{equation}
\label{eq:abn}
\langle \lambda(t),\D_{\pi(\lambda(t))}\rangle=0\quad \text{and} \quad \lambda(t)\neq 0,\qquad \forall\, t\in[0,T] ,
\end{equation}
that is $H(\lambda(t))\equiv 0$, therefore abnormal geodesics are always contained in the singular region $\Z$. A geodesic may be abnormal and normal at the same time. 

\begin{defn}
Let $M$ be an \ariem manifold and let $p\in M$. Then, the \emph{almost-Riemannian exponential map} is 
\begin{equation}
\label{eq:ar_exp_map}
\srexp_p(\lambda)=\pi\circ e^{\vec H}(\lambda),\qquad\forall\lambda\in T_p^*M,
\end{equation} 
where $H$ denotes the \ariem Hamiltonian \eqref{eq:Hamiltonian} and $e^{\vec H}(\lambda)$ is the solution to \eqref{eq:Hamiltoneqs} at time $t=1$, with initial datum $\lambda\in T_p^*M$.  
\end{defn}

Note that, in general, $\srexp_p$ may not be defined on the whole cotangent space, but if $M$ is complete, then $\vec{H}$ is a complete vector field and \eqref{eq:ar_exp_map} is well-posed.

\subsection{Length-minimizers to a hypersurface} 
Let $\surf\subset M$ be a smooth hypersurface and fix $q_0\in \surf$. Moreover, let $v\in C^\infty(M)$ be a local defining function for $\surf$ around $q_0$, namely there exists an open neighborhood $\surf_{q_0}\subset\surf$ of $q_0$ such that
\begin{equation}
\label{eq:local_def_fun}
    \surf_{q_0}\subset\{v=0\}\qquad\text{and}\qquad dv_{|\surf_{q_0}}\neq 0.    
\end{equation}
We define the \emph{local signed distance function} from $\surf$ around $q_0$  as follows:
\begin{equation}
\label{eq:dist_surf}
 \delta_v:= {\rm sgn}(v(p)) \cdot\di(p,\{v=0\}),\qquad\forall\,p\in M.
\end{equation}
%
%
%
%
Let $\gamma\colon[0,T]\rightarrow M$ be a horizontal curve, parameterized with constant speed, such that $\gamma(0)\in \surf$, $\gamma(T) = p \in M\setminus \surf$ and assume $\gamma$ is a minimizer for $\di(\cdot,\surf)$, that is $\ell(\gamma)=\di(p,\surf)$. In particular, $\gamma$ is a geodesic and any corresponding normal or abnormal lift, say $\lambda :[0,T]\to T^*M$, must satisfy the transversality conditions, cf.\ \cite[Thm 12.13]{AS-GeometricControl},
\begin{equation}\label{eq:trcondition}
\langle \lambda(0), w\rangle=0,\qquad \forall \,w\in T_{\gamma(0)} \surf.
\end{equation}
Equivalently, the initial covector $\lambda(0)$ must belong to the \emph{annihilator bundle} $\A\surf$ of $\surf$ with fiber $\A_q\surf= \{\lambda \in T_q^*M \mid \langle \lambda, T_q \surf\rangle = 0\}$, for any $q\in\surf$. The restriction of $\srexp_q$ to the annihilator bundle of $\surf$ allows to build (locally) a smooth tubular neighborhood around non-characteristic points. Recall that $q\in\surf$ is a \emph{characteristic point}, and we write $q\in C(\surf)$, if $\D_q\subset T_q\surf$.

\begin{lem}
\label{lem:exp_sigma}
Let $\surf\subset M$ be a smooth hypersurface, let $q_0\in\surf\setminus C(\surf)$ be a non-characteristic point and $v\in C^\infty(M)$ as in \eqref{eq:local_def_fun}. Then, there exist $\eps_{q_0}>0$ and a neighborhood $\mathcal{O}_{q_0}\subset\surf_{q_0}$ of $q_0$ such that the map 
\begin{equation}
\label{eq:diffeo_G}
G\colon (-\eps_{q_0},\eps_{q_0})\times\mathcal{O}_{q_0}\rightarrow M,\qquad G(s,q)=\exp_q(s\lambda(q)),
\end{equation} 
is a diffeomorphism on its image, where $\lambda(q)$ is the unique element (up to a sign) of $\A_q\surf$ such that $2H(\lambda(q))=1$. Moreover, $\delta_v$ is smooth in $G((-\eps_{q_0},\eps_{q_0})\times\mathcal{O}_{q_0})$ and\footnote{The \emph{horizontal gradient} of $f\in C^\infty(M)$ is defined by $g_p(\nabla f,v)=d_pf(v),\forall\,v\in\D_p$ and $p\in M$.}
\begin{equation}
\label{eq:grad_delta}
{G_*\partial_s}_{|(s,q)}=\nabla\delta_v(G(s,q)),\qquad\forall\,(s,q)\in (-\eps_{q_0},\eps_{q_0})\times\mathcal{O}_{q_0}.
\end{equation}
\end{lem}

\begin{rmk}
It is known that if $\surf$ has no characteristic points, the signed distance is smooth in a tubular neighborhood of $\surf$, cf. \cite[Prop.\ 3.1]{FPR-sing-lapl}. This lemma can be regarded as its local version and its proof is a straightforward adaptation of the aforementioned result. Moreover, note that $C(\surf)\subset\Z$ and so the Riemannian points of $\surf$ are non-characteri\-stic. Finally, if $\surf$ contains characteristic points, the parameter $\eps_{q_0}$, as well as $\mathcal{O}_{q_0}$, can not be chosen uniformly. 
\end{rmk}

\begin{rmk}
\label{rmk:min_geod}
By condition \eqref{eq:grad_delta}, for any $q\in\mathcal{O}_{q_0}$, we have 
\begin{equation}
\label{eq:min_geod}
(-\eps_{q_0},\eps_{q_0})\in s\mapsto G(s,q)\in M
\end{equation}
is the unique minimizing geodesic (parameterized by unit-speed) from $\surf$ passing through $q$. Moreover, notice that the initial covector $\lambda(q)$ in \eqref{eq:diffeo_G} is unique up to a sign: the only requirement is to choose this covector in such a way it defines a continuous section of the annihilator bundle.
\end{rmk}

\section{The curvature-dimension condition}\label{sec:cd}

A triple $(X,\di,\m)$ is called metric measure space if $(X,\di)$ is a complete and separable metric space and $\m$ is a locally finite Borel measure on it. In the following $C([0, 1], X)$ will stand for the space of continuous curves from $[0, 1]$ to $X$. A curve $\gamma\in C([0, 1], X)$ is called \textit{minimizing geodesic} if 
\begin{equation}
    \di(\gamma_s, \gamma_t) = |t-s| \cdot  \di(\gamma_0, \gamma_1) \quad \text{for every }s,t\in[0,1],
\end{equation}
we denote by $\Geo(X)$ the space of minimizing geodesics on $X$. The metric space $(X,\di)$ is said to be geodesic if every pair of points $x,y \in X$ can be connected with a curve $\gamma\in \Geo(X)$. 
For any $t \in [0, 1]$ we define the evaluation map $e_t \colon C([0, 1], X) \to X$ by setting $e_t(\gamma) := \gamma_t$ and the stretching/restriction operator $\rest{r}{s}$ in $C([0, 1], X)$, defined, for all $0\leq r<s\leq1$, by
\begin{equation}
[\rest{r}{s}(\gamma)]_t := \gamma_{r + t(s-r)}, \qquad t \in [0, 1].
\end{equation}
We denote by $\Prob(X)$ the set of Borel probability measures on $X$ and by $\Prob_2(X) \subset \Prob(X)$ the set of those having finite second moment. We endow the space $\Prob_2(X)$ with the Wasserstein distance $W_2$, defined by
\begin{equation}
\label{eq:defW2}
    W_2^2(\mu_0, \mu_1) := \inf_{\pi \in \mathsf{Adm}(\mu_0,\mu_1)}  \int \di^2(x, y) \, \de \pi(x, y),
\end{equation}
where $\mathsf{Adm}(\mu_0, \mu_1)$ is the set of all the admissible transport plans between $\mu_0$ and $\mu_1$, namely all the measures in $\Prob(X^2)$ such that $(\p_1)_\sharp \pi = \mu_0$ and $(\p_2)_\sharp \pi = \mu_1$. The metric space $(\Prob_2(X),W_2)$ is itself complete and separable, moreover, if $(X,\di)$ is geodesic, then $(\Prob_2(X),W_2)$ is geodesic as well. In particular, every geodesic $(\mu_t)_{t\in [0,1]}$ in $(\Prob_2(X),W_2)$ can be represented with a measure $\eta \in \Prob(\Geo(X))$, meaning that $\mu_t = (e_t)_\# \eta$. A subset $G\subset \Geo(X)$ is called non-branching if for any pair $\gamma_1,\gamma_2\in G$ such that $\gamma_1\neq\gamma_2$, it holds that 
 \begin{equation}
     \rest{0}{t} (\gamma_1) \neq  \rest{0}{t} (\gamma_2) \qquad \text{for every }t\in (0,1).
 \end{equation}
 A metric measure space $(X,\mathsf{d},\mathfrak{m})$ is said to be essentially non-branching if for every two measures $\mu_0,\mu_1\in\Prob_2(X)$ which are absolutely continuous with respect to the reference measure $\m$ ($\mu_0,\mu_1 \ll \m$), every $W_2$-geodesic connecting them is concentrated on a non-branching set of geodesics.

\subsection{\texorpdfstring{$\cd$}{CD} spaces}

In this subsection we introduce the $\cd$ condition, pioneered by Sturm and Lott-Villani \cite{MR2237206,MR2237207,MR2480619}. This condition aims to generalize, to the context metric measure spaces, the notion of having Ricci curvature bounded from below by $K$ and dimension less than or equal to $N$. In particular, in the Riemannian setting it is possible to characterize this two bounds in terms of a property whose definition involves only the distance and the (volume) measure. This property, which is stated in Definition \ref{def:CD}, is given in terms of the following distortion coefficients: for every $K \in \R$ and $N\in (1,+\infty)$
\begin{equation}
    \tau_{K,N}^{(t)}(\theta):=t^{\frac{1}{N}}\left[\sigma_{K, N-1}^{(t)}(\theta)\right]^{1-\frac{1}{N}},
\end{equation}
where
\begin{equation}
\sigma_{K,N}^{(t)}(\theta):= 
\begin{cases}

\displaystyle  \frac{\sin(t\theta\sqrt{K/N})}{\sin(\theta\sqrt{K/N})} & \textrm{if}\  N\pi^{2} > K\theta^{2} >  0, \crcr
t & \textrm{if}\ 
K =0,  \crcr
\displaystyle   \frac{\sinh(t\theta\sqrt{-K/N})}{\sinh(\theta\sqrt{-K/N})} & \textrm{if}\ K < 0.
\end{cases}
\end{equation}

\begin{defn}\label{def:CD}
A metric measure space $(X,\di,\m)$ is said to be a $\cd(K,N)$ space (or to satisfy the $\cd(K,N)$ condition) if for every pair of measures $\mu_0=\rho_0\m,\mu_1= \rho_1 \m \in \Prob_2(X)$, absolutely continuous with respect to $\m$, there exists a $W_2$-geodesic $(\mu_t)_{t\in [0,1]}$ connecting them and induced by $\eta \in \Prob(\Geo(X))$, such that for every $t\in [0,1]$ $\mu_t =\rho_t \m \ll \m$ and the following inequality holds for every $N'\geq N$ and every $t \in [0,1]$
\begin{equation}\label{eq:CDcond}
    \int_X \rho_t^{1-\frac 1{N'}} \de \m \geq \int_{X \times X} \Big[ \tau^{(1-t)}_{K,N'} \big(\di(x,y) \big) \rho_{0}(x)^{-\frac{1}{N'}} +    \tau^{(t)}_{K,N'} \big(\di(x,y) \big) \rho_{1}(y)^{-\frac{1}{N'}} \Big]    \de\pi( x,y),
\end{equation}
where $\pi= (e_0,e_1)_\# \eta$.
\end{defn}

In general, the $\cd$ condition is not very easy to disprove, however when the reference space is an interval $I \subseteq \R$ the following lemma, whose proof can be find in \cite[Lemma A.5]{MR4309491}, provides a nice strategy.

\begin{lem}
\label{lem:diff_char}
Let $I\subset \R$  be an interval and let $h:I \to \R$ be a measurable function such that $(I,|\cdot|, h \Leb^1)$ is a $\cd(K,N)$ space. Then at any point $x$ in the interior of $I$ where $h$ is twice differentiable it holds that 
\begin{equation}
\label{eq:1d_characterization}
    (\log h)^{\prime \prime}(x)+\frac{1}{N-1}\left((\log h)^{\prime}(x)\right)^{2} \leq-K.
\end{equation}
\end{lem}

\begin{rmk}
This lemma holds also for $N=+\infty$, where now the left-hand side of \eqref{eq:1d_characterization} has to be intended as $(\log h)^{\prime \prime}(x)$ (for the definition of $\cd(K,\infty)$ space, see \cite{MR2237206}).
\end{rmk}

\noindent In fact, in order to disprove that the space $(I,|\cdot|, h \Leb^1)$ satisfies $\cd(K,N)$ is sufficient to find a point $x$ in the interior of $I$ such that $h$ is twice differentiable in $x$ and
\begin{equation}
    (\log h)^{\prime \prime}(x)+\frac{1}{N-1}\left((\log h)^{\prime}(x)\right)^{2} >-K.
\end{equation}
Notice also that, if we manage to prove that 
\begin{equation}\label{eq:disproveCD}
    (\log h)^{\prime \prime}(x) >-K,
\end{equation}
we automatically show that $(I,|\cdot|, h \Leb^1)$ does not satisfy $\cd(K,N)$ for every $N \in (1,+\infty]$. This observation will be fundamental in the following, especially in combination with the one-dimensional localization results we are now going to present.

\subsection{One-dimensional localization}
\label{subsection:onedimCD}
In this subsection we present a suitable adaptation of the one-dimensional characterization of the $\cd$ condition. This property, called $\cd^1(K,N)$ condition, has been studied in the general framework of essentially non-branching metric measure spaces with a curvature-dimension bound in \cite{cavalletti2013,MR3648975,MR4175820,MR4309491}. 
We provide a local version of such characterization, that allows us to take advantage of the local structure of \ariem manifolds. 

We recall a general result regarding disintegration of measures. Given a measurable space $(R,\mathscr R)$, and a function $\Q : R \to Q$ to a general set $Q$, we endow $Q$ with the push forward $\sigma$-algebra $\mathscr Q$ of $\mathscr R$, i.e. the biggest $\sigma$-algebra on $Q$ such that $\Q$ is measurable. Moreover, given a finite (non-null) measure $\rho$ on $(R,\mathscr R)$, consider the measure $\q:= \Q_\# \rho$ on $(Q,\mathscr Q)$.

\begin{defn}
    A disintegration of $\rho$ consistent with $\mathfrak Q$ is a map $Q\ni q\mapsto \rho_q \in \Prob(R)$ such that the following hold:
    \begin{enumerate}
        \item for all $B \in \mathscr R$, $\rho_\cdot (B)$ is $\q$-measurable,
        \item for all $B \in \mathscr R$, $C \in \mathscr Q$, we have
    \begin{equation*}
        \rho\left(B \cap \mathfrak{Q}^{-1}(C)\right)=\int_{C} \rho_{q}(B) \, \de \q( q).
    \end{equation*}
    \end{enumerate}
     A disintegration is called \emph{strongly consistent} with respect to $\Q$ if, in addition, for all $q \in Q$ it holds that $\rho_q({\mathfrak Q}^{-1}(q)) =1$.
\end{defn}

\begin{thm}[{\cite[Thm.\ 2.8]{MR3648975}}]\label{thm:disintegration}
    Let $(R,\mathscr R)$ be a countably generated measurable space and 
    $\rho$ be a finite measure on it. Assume there exists a partition of $R$ as
    \begin{equation*}
        R = \bigcup_{q \in Q} R_q,
    \end{equation*}
 denote by $\Q : R \to Q$ the quotient map and by $(Q, \mathscr Q, \q)$ the quotient measure space. If $(Q, \mathscr Q) = (X, \mathscr B(X))$ where $X$ is a Polish space and $\mathscr B(X)$ denotes its Borel $\sigma$-algebra, then there exists a unique strongly consistent disintegration $q \mapsto \rho_q$ with respect to $\Q$.
\end{thm}

Let $(X,\di,\m)$ be a metric measure space and fix an open subset  $\Omega\subset X$ with $0< \m (\Omega) = \m(\Bar{\Omega})<\infty$. Let $u:\Bar{\Omega} \to \R$ be a $1$-Lipschitz function, define
\begin{equation}
    \Gamma_u:= \{(x,y)\in \Bar{\Omega}\times \Bar{\Omega} \mid  u(x)-u(y)= \di(x,y)\}
\end{equation}
and its transpose $\Gamma_u^{-1}:=\{(x,y)\in \Bar{\Omega}\times \Bar{\Omega}\mid (y,x)\in \Gamma_u\}$.
Consequently, we introduce the transport relation $R_u$ and the transport set $T_u$ as 
\begin{equation}
    R_u:= \Gamma_u \cup \Gamma_u^{-1} \quad \text{and} \quad T_u := \p_1 \big(R_u \setminus \{(x,y)\in \Bar{\Omega}\times \Bar{\Omega} \mid x=y\}\big),
\end{equation}
where $\p_1$ denotes the projection on the first factor. Although this is not always the case, if we assume that $R_u$ is an equivalence relation, we may partition the set $\Bar{\Omega}$. Letting $Q$ be the set of equivalence classes and $\Q:\bar\Omega \to Q$ the quotient map, we can write 
\begin{equation*}
    \Bar{\Omega} = \bigcup_{q \in Q} \gamma_q,
\end{equation*}
where $\gamma_q:=\{x\in \Bar{\Omega} \mid \Q (x) = q\}$ for every $q \in Q$.
With the quotient map we can endow $Q$ with the quotient $\sigma$-algebra $\mathscr Q$, that is the finest $\sigma$-algebra on $Q$ for which $\Q$ is measurable. We introduce the following definition to obtain a local version of the one-dimensional localization of \cite{MR3648975}, which better fits the setting of almost-Riemannian geometry, where we have a good local description of geodesics. In \cite{MR3648975}, the authors define a global partition starting from a globally defined $1$-Lipschitz function, see Remark \ref{rmk:why_we_need_this_comment}. 

\begin{defn}
\label{def:1d_partition}
We say that a $1$-Lipschitz function $u:\Bar{\Omega} \to \R$ induces a \emph{one-dimensio\-nal partition} of $\Bar{\Omega}$ if 
\begin{enumerate}
    \item $R_u$ is an equivalence relation and $\m(\Bar{\Omega} \setminus T_u)=0$,
    \item for every $q\in Q$, the set $\gamma_q \subset \bar\Omega$ is the image of a geodesic of $(X,\di)$,
    \item for every $q\in Q$ there exists $x\in \bar\Omega$ such that $\Q(x)=q$ and $u(x)=0$.
\end{enumerate}
\end{defn}

\noindent If $u$ induces a one-dimensional partition, then, in particular, we can choose $Q=\{ u=0\}$. Indeed, the point $x\in \bar\Omega$ satisfying  $(3)$ of Definition \ref{def:1d_partition} is unique. Define the ray map 
\begin{equation}
    g: \operatorname{Dom}(g) \subset Q \times \mathbb{R} \rightarrow \Bar{\Omega}
\end{equation}
by imposing that 
\begin{equation*}
    \begin{aligned}
\operatorname{graph}(g):=  \{(q, t, x) \in Q \times\R \times \bar\Omega\mid \Q(x)=q , u(x)=t\}.
\end{aligned}
\end{equation*} 
The ray map $g$ is Borel and bijective, its inverse is
 \begin{equation}
   \Bar{\Omega} \ni  x \mapsto g^{-1}(x) := (\Q(x), u(x)).
 \end{equation}
Moreover, for every $q\in Q$ the map $t \mapsto g(q,\cdot)$ is an isometry, and consequently $\mathcal{H}^{1}\mres \gamma_q = \mathcal{H}^{1}\mres\{g(q, t)\mid t \in I_q\}=g(q, \cdot)_{\#} \Leb^{1}$, where $I_q:=\operatorname{Dom}(g(q, \cdot))$. Theorem \ref{thm:disintegration} ensures that there exists a unique strongly consistent disintegration of $\m\mres{\Omega}$:
\begin{equation*}
    \m\mres{\Omega} = \int_Q \m_q \, \de \q (q),
\end{equation*}
where $\m_q$ is a measure concentrated on $\gamma_q$ and we recall that $\q:= \mathfrak Q_\#(\m\mres{\Omega})$.

\begin{prop}\label{prop:localization}
     Let $(X, \di,\m)$ be a essentially non-branching metric measure space satisfying the $\cd(K,N)$ condition, for some $K\in \R$ and $N \in (1,\infty)$. Let $\Omega\subset X$ be open and such that $0< \m (\Omega) = \m(\Bar{\Omega})<\infty$ and let $u:\Bar{\Omega} \to \R$ be a $1$-Lipschitz function providing a one-dimensional partition. Then:
     \begin{enumerate}
         \item for $\q$-a.e. $q \in Q$, the measure $\m_q$ is absolutely continuous with respect to $\mathcal{H}^{1}\mres\gamma_q$, namely there exists $h_q: I_q\rightarrow [0,\infty]$ such that $\mathfrak{m}_{q}=g(q, \cdot)_{\#}\left(h_{q} \cdot \Leb^{1}\right)$,
         \item for $\q$-a.e. $q\in Q$, $(I_q,|\cdot|,h_q\Leb^1)$ is a $\cd(K,N)$ space. 
     \end{enumerate}
\end{prop}
\begin{proof}
The proof of this proposition can be done by adapting the classical global approach to the space $(\bar\Omega, \di ,\m \mres \Omega)$, see in particular \cite[Sec.\ 6]{cavalletti2013} for point (1) and \cite[Thm.\ 4.2]{MR3648975} for point (2). We point out that this space is not necessarily geodesic, hence we can not conclude that it satisfies the $\cd(K,N)$ condition and simply apply the known results. However, observe that, in order to deduce properties of the disintegration induced by $u$, it is enough to study Wasserstein geodesics that follow its transport rays. Since $u$ induces a one-dimensional partition in $\bar\Omega$ in the sense of Definition \ref{def:1d_partition}, all the transport rays are contained in $\bar\Omega$ and the $\cd$ condition \eqref{eq:CDcond} holds along such Wasserstein geodesics. For this reason, we can repeat the standard arguments verbatim, obtaining the result. 
\end{proof}

\begin{rmk}
\label{rmk:why_we_need_this_comment}
    
In the classical theory of \cite{MR3648975}, starting from a \emph{globally defined} $1$-Lipschitz function on a $\cd(K,N)$ essentially non-branching metric measure space, the authors build a one-dimensional partition of the whole space, up to a negligible set, and disintegrate the measure accordingly. Then, the densities in the disintegration satisfy the $\cd(K,N)$ condition, providing the one-dimensional characterization $\cd^1(K,N)$.  In this setting, there is no need for the additional properties of Definition \ref{def:1d_partition} on the $1$-Lipschitz function $u$. 

In particular, given a $1$-Lipschitz function $u:X \to \mathbb R$, we can introduce the transport relation $R_u$ and the transport set $T_u$ as before (with $X$ in place of $\bar\Omega$) and denote by $\Gamma_u(x)$ the section of $\Gamma_u$ through $x$ in the first coordinate. Then, we define the set of forward and backward branching points as 
\begin{equation}
    \begin{split}
    A^+ &:=\{x \in T_u \mid \exists y,z \in \Gamma_u(x), (y,z)\notin R_u\}\\
    A^- &:=\{x \in T_u \mid \exists y,z \in \Gamma_u^{-1}(x), (y,z)\notin R_u\}.
    \end{split}
\end{equation}
Finally, we define the non-branched transport set and the non-branched transport relation as
\begin{equation}
    T_u^{nb}:= T_u \setminus (A^+ \cup A^-), \qquad\text{and}\qquad R_u^{nb}:= R_u \cap ( T_u^{nb} \times  T_u^{nb}).
\end{equation}
On the one hand, as shown in \cite{cavalletti2013}, the essentially non-branching assumption ensures that $R_u^{nb}$ is an equivalence relation on $T_u^{nb}$ and for $\q$-a.e. $q \in Q$, $\gamma_q$ is isometric to a closed interval of $\R$. On the other hand, if $(X,\di,\m)$ also satisfies the $\cd(K,N)$ condition, the set $T_u \setminus T_u^{nb}$ is $\m$-negligible, cf. \cite[Thm.\ 3.4]{MR3648975}. It is then possible to obtain a global result, analogous to Proposition \ref{prop:localization}. 
\end{rmk}

\begin{rmk}
Note that showing \eqref{eq:disproveCD} for $K\in\R$ actually implies that $h_q$ can not be a $\cd(K,\infty)$ density. However, for a metric measure space $(X,\di,\m)$, it is not known whether the $\cd(K,\infty)$ condition can be characterized with one-dimensional disintegrations.
\end{rmk}





\section{A general strategy for disproving the \texorpdfstring{$\cd$}{CD} condition}
\label{sec:gen_strategy}

The $1$-Lipschitz function whose disintegration allows us to disprove the $\cd$ condition will be a localized version of the (signed) distance function from a hypersurface $\surf$. Indeed with this choice we are able to compute explicitly the one-dimensional marginals and to exploit the existence of characteristic points.

\subsection{Existence of normal coordinates}

Let $M$ be an \ariem manifold of dimension $n+1$. We build a convenient set of coordinates around a point in the singular region. This result can be regarded as a generalization of \cite[Prop.\ 9.8]{ABB-srgeom}.

\begin{lem}[Normal coordinates]
\label{lem:norm_coord}
Let $M$ be an \ariem manifold of dimension $n+1$, let $\Z\subset M$ the set of singular points and let $p_0\in\Z$. Then, there exists a set of coordinates $\varphi=(x,z)\colon\mathcal{U}\rightarrow M$ centered at $p_0$ and a local orthonormal frame $X_0, \dots , X_n$ for the almost-Riemannian structure on $\mathcal{U}$ such that:
    \begin{equation}
        \label{eq:loc_frame}
        X_0=\partial_x,\qquad X_i=\sum_{j=1}^na_{ij}(x,z)\partial_{z_j},\quad\forall\,i=1,\ldots,n,
    \end{equation}
where $a_{ij}$ are smooth functions in $\mathcal{U}$. Moreover, denoting by $A(x,z)=(a_{ij}(x,z))_{i.j}$, we have $\det A(0,\zero)=0$ .
\end{lem}

\noindent The proof of Lemma \ref{lem:norm_coord} follows from the existence of a tubular neighborhood around non-characteristic points. We require a preliminary result.

\begin{lem}
\label{lem:nonchar_surf}
Let $M$ be an \ariem manifold and let $p_0\in M$. Then, there exists a hypersurface $W\subset M$ such that $p_0\in W\setminus C(W)$.
\end{lem}

\begin{proof}
We assume by contradiction that $p_0\in M$ is a characteristic point for any hypersurface $W\subset M$ passing through $p_0$. By definition of characteristic point, this means that $\D_{p_0}\subset T_{p_0}W$ for any such $W$. In turn, this implies that $\D_{p_0}=\{0\}$, contradicting the bracket-generating assumption.
\end{proof}


\begin{proof}[Proof of Lemma \ref{lem:norm_coord}]
Using Lemma \ref{lem:nonchar_surf}, we find an embedded hypersurface $W\subset M$ such that $p_0\in W\setminus C(W)$. We use the \ariem normal exponential map to define the desired coordinates. Indeed, let $v\in C^\infty(M)$ be a local defining function as in \eqref{eq:local_def_fun}. Then, by Lemma \ref{lem:exp_sigma}, there exist $\eps_{p_0}>0$ and a neighborhood $\mathcal{O}_{p_0}\subset W$ of $p_0$ such that 
\begin{equation}
G\colon (-\eps_{p_0},\eps_{p_0})\times\mathcal{O}_{p_0}\rightarrow M,\qquad G(s,p)=\exp_p(s\lambda(p)),
\end{equation} 
is a diffeomorphism on its image, where $\lambda(p)$ satisfies \eqref{eq:trcondition} with $2H(\lambda(p))=1$. Moreover, the local signed distance function $\delta_v$ is smooth in $G((-\eps_{p_0},\eps_{p_0})\times\mathcal{O}_{p_0})$ and
\begin{equation}
{G_*\partial_s}_{|(s,p)}=\nabla \delta_v(G(s,p)),\qquad\forall\,(s,p)\in (-\eps_{p_0},\eps_{p_0})\times\mathcal{O}_{p_0}.
\end{equation}
Thus, fixing any set of coordinates $(z_1,\ldots,z_n)$ for $\mathcal{O}_{p_0}$, and relabelling $s=x$, the coordinates $(x,z)$ satisifies \eqref{eq:loc_frame}. Finally, since $p_0\in\Z$, the vector fields $\{X_0,\ldots,X_n\}$ are linearly dependent at $p_0$, meaning that the matrix $A(x,z)$ has zero determinant at $(0,\zero)$.  
\end{proof}

\begin{rmk}
    From now on, without loss of generality, whenever we fix a set of coordinates, we will assume that the domain of the chart is the whole $\R\times\R^n$. 
\end{rmk}

\subsection{Assumptions on the \ariem structure}
\label{subsection:assumption}
Let $M$ be an \ariem manifold of dimension $n+1$, and let $\Z\subset M$ be the set of singular points. Let $p_0\in\Z$ and let $\surf\subset M$ be a hypersurface. To proceed with our general construction, we need two assumptions on the \ariem structure:  in coordinates $(x,z)\in\R\times\R^n$ centered at $p_0$ given by Lemma \ref{lem:norm_coord}, we require that

\begin{enumerate}[label=\roman*)]
    \item the hypersurface $\surf$ consists of Riemannian points except when $x=0$ and has a characteristic point at the origin, i.e. 
\begin{equation}
\label{eq:cond_hypsurf}\tag{$\mathbf{H1}$}
    \surf\cap\Z\subset \{x=0\}\quad\text{and}\quad (0,\zero)\in C(\surf);
\end{equation}
    \item let $\m$ be any smooth positive measure on $M$, then 
    \begin{equation}
    \label{eq:cond_measure}\tag{$\mathbf{H2}$}
    \m(\Z)=0.
    \end{equation}
\end{enumerate}
Here by smooth positive measure, we mean a measure with strictly positive and smooth density with respect to the Lebesgue measure in coordinates. 

\begin{rmk}
Let us comment on why we need these assumptions. The first one is necessary to have a good local description of the marginals \emph{outside} the set $\{x=0\}$, in order to exploit the presence of a characteristic point \emph{only} at the origin, cf.\ Lemma \ref{lem:exp_sigma}. The second one is necessary in order to ensure the essentially non-branching property, cf.\ Lemma \ref{lem:enb}, and to characterize the marginals in the disintegration, cf.\ \eqref{eq:density}.
\end{rmk}

\begin{rmk}
We remark that assumption \eqref{eq:cond_measure} is not always guaranteed as the next example shows. Consider $C\subset [0,1]$ a closed subset with positive Lebesgue measure and empty interior and let $f\in C^\infty(\R)$ such that $f_{|C}\equiv 0$. Then, define the structure on $\R^4$ with global orthonormal frame: 
\begin{equation}
    X_0=\partial_x,\qquad X_1=\partial_{z_1}-\frac{z_2}{2}\partial_{z_3},\qquad X_2=\partial_{z_2}+\frac{z_1}{2}\partial_{z_3},\qquad X_3=f(x)\partial_{z_3}.
\end{equation}
As one can check, the H\"ormander condition is verified and the local minimal bundle rank is always $4$, since $C$ has empty interior. Thus, the structure is almost-Riemannian. Now fix $\m=\Leb^4$, then the singular set has infinite measure indeed, by \eqref{eq:singular_set}
\begin{equation}
    \Z=\{(x,z)\in\R^4\mid \det A(x,z)=0\}=\{(x,z)\in\R^4\mid f(x)=0\}=C\times\R^3.
\end{equation}
Here, the matrix $A$ is given by
\begin{equation}
    A(x,z)=
    \begin{pmatrix}
    1& & \\
     &1& \\
    -\frac{z_2}{2}&\frac{z_1}{2}&f
    \end{pmatrix}.
\end{equation}
In an analogue way, one can build an example where \eqref{eq:cond_hypsurf} is not verified. Indeed, in the construction above, it is enough to consider a closed set $C$ with empty interior and with an accumulation point at the origin. Then, the hypersurface $\surf=\{z_3=0\}$ has a characteristic point at the origin, however it intersects the singular region in $C\times\R^2$. 
\end{rmk}

Notice that, in coordinates \eqref{eq:loc_frame}, the singular set can be described by the matrix $A=(a_{ij})_{i,j}$, indeed
\begin{equation}
\label{eq:singular_set}
(x,z)\in\Z\quad\text{if and only if}\quad\det A(x,z)=0.
\end{equation}
In particular, along the hypersurface $\surf$, by \eqref{eq:cond_hypsurf}, we have 
\begin{equation}
\label{eq:riem_points_surf}
    x\neq 0\quad\Rightarrow\quad  \det A(x,z_1,\ldots,z_{n-1},0)\neq 0,
\end{equation}
since the set $\surf\cap\{x\neq 0\}$ consists of Riemannian points. As a consequence, since a Riemannian point is never a characteristic one, $C(\surf)\subset \{x=0\}$. Actually, it is always possible to ensure that $\surf=\{z_n=0\}$ satisfies $ (0,\zero)\in C(\surf)$, so the only condition one should check is \eqref{eq:riem_points_surf}.

\begin{lem}
\label{lem:nice_hypersurf}
Let $M$ be an \ariem manifold and let $p_0\in\Z$. Then, there exists an hypersurface $\surf\subset M$ such that $p_0\in C(\surf)$. Moreover, in coordinates $(x,z)$ as in \eqref{eq:loc_frame}, up to a rotation, we can choose $\surf=\{z_n=0\}$. 
\end{lem}

\begin{proof}
Assume by contradiction that $p_0\in M$ is not a characteristic point for every hypersurface $W\subset M$ passing through $p_0$. Then, by definition of characteristic point, we deduce that $\D_{p_0}$ must be transversal to $T_{p_0}W$, for every such $W$, or equivalently
\begin{equation}
\D_{p_0}+T_{p_0}W=T_{p_0}M,     
\end{equation}
for every $W\subset M$ passing trough $p_0$. As a consequence, $\D_{p_0}=T_{p_0}M$ and thus $r(p_0)=n+1$. This gives a contradiction, since $r(p_0)<n+1$, being $p_0\in \Z$. Let us show that in coordinates $\surf$ can be chosen as $\{z_n=0\}$: since $\det A(p_0)=0$, there exists an invertible matrix $M\in \mathrm{GL}(n+1,\R)$ such that the last column of the matrix $A(p_0)M$ consists of zeroes. Then, we introduce the following change of coordinates
\begin{equation}
\psi\colon (x,z)\mapsto (x,\tilde z)=(x,M^\intercal z).
\end{equation} 
In the new coordinates, the generating family for the distribution has the following expression:
\begin{equation}
X_0=\partial_x,\qquad X_i=\sum_{j=1}^n a_{ij}(x,z)\partial_{z_j}=\sum_{k,j=1}^n a_{ij}(x,\psi^{-1}(z))m_{jk}\partial_{\tilde z_k},
\end{equation} 
having denoted by $M=(m_{ij})_{i,j}$. Thus, \eqref{eq:loc_frame} is still valid and, when evaluated at $p_0$, the matrix describing the generating family has the last column consisting of zeroes. Finally, this implies that the hypersurface $\surf=\{\tilde z_n=0\}$ has a characteristic point at $p_0$. Indeed, 
\begin{equation}
\nabla \tilde z_n(p_0)=\sum_{i=1}^n X_i(\tilde z_n)X_i(p_0)=\sum_{i,j=1}^n a_{ij}(0,\zero)m_{jn}X_i(p_0)=\sum_{i=1}^n (A(p_0)M)_{in}X_i(p_0)=0,
\end{equation}
since the last column of the matrix $A(p_0)M$ is zero, implying that $\D_{p_0}\subset T_{p_0}\surf$. 
\end{proof}

\begin{lem}\label{lem:enb}
Let $M$ be an \ariem manifold, equipped with a smooth positive measure $\m$ and satisfying assumptions \eqref{eq:cond_measure}. Then $(M,\di,\m)$ is essentially non-branching.
\end{lem}

\begin{proof}
Let $\gamma\colon[0,1]\rightarrow M$ be a minimizing geodesic. Then $\gamma$ is abnormal if there exists an abnormal extremal lift $\lambda(t)\neq 0$, satisfying \eqref{eq:abn}. But this implies that
\begin{equation}
    \gamma(t)\subset\Z,\qquad\forall\,t\in [0,1].
\end{equation}
Hence, if $\gamma$ is a minimizing geodesic with endpoints in the Riemannian region, i.e.
\begin{equation}
\label{eq:endpoints}
    \gamma(0),\gamma(1)\in M\setminus\Z,
\end{equation}
then $\gamma$ must be strictly normal. As showed in \cite[Cor.\ 6]{MR4153911}, a strictly normal geodesic $\gamma \colon[0,1]\rightarrow M$ is branching for some positive time $t\in (0,1)$ if and only if it contains a non-trivial abnormal subsegment that starts at time $0$. Thus, a minimizing geodesic satisfying \eqref{eq:endpoints} can not branch for positive times since $\Z$ is closed. Now, let $\eta\in \Prob(\Geo(X))$ be a $W_2$-geodesic joining the measures $\mu_0,\mu_1\in \Prob_2(X)$, which are absolutely continuous with respect to the reference measure $\m$ ($\mu_0,\mu_1 \ll \m$). In particular, notice that $(e_0)_\# \eta = \mu_0$ and $(e_1)_\# \eta = \mu_1$ and therefore, by \eqref{eq:cond_measure},
\begin{equation}
    \eta(e_0^{-1}(\Z))= \mu_0(\Z)= 0 \quad \text{ and } \quad \eta(e_1^{-1}(\Z))= \mu_1(\Z)= 0 .
\end{equation}
Consequently, the measure $\eta$ is concentrated on $\Geo(X) \setminus \big(e_0^{-1}(\Z)  \cup e_1^{-1}(\Z)\big)$, which is a non-branching set of geodesics, according to the first part of the proof.
\end{proof}

\begin{rmk}
Notice that it is possible to build examples of almost-Riemannian manifolds where \eqref{eq:cond_measure} is verified but there exist branching geodesics. Indeed, consider $\R^4$, with the global orthonormal frame given by 
\begin{equation}
    X_0=\partial_x,\qquad X_1=\partial_{z_1},\qquad X_2=\partial_{z_2}+B(z_1,z_2)\partial_{z_3},\qquad X_3=x\partial_{z_3},
\end{equation}
where $B$ is a smooth magnetic potential, defined as in \cite{MR4153911}, namely
\begin{equation}
    B(z_1,z_2)=z_1\theta(z_2)+z_2^2\theta(1-z_2)
\end{equation}
and $\theta\in C^\infty(\R)$ such that $0\leq\theta\leq 1$, $\theta(r)=0$ for $r\leq 0$ and $\theta(r)=1$ for $r\geq 1$. In this situation, we have strictly normal branching geodesic in the singular region. Nevertheless, thanks to Lemma \ref{lem:enb}, $(M,\di,\m)$ is essentially non-branching. On the other hand, if the measure of the singular set is positive, it is unclear whether an \ariem manifold is essentially non-branching.
\end{rmk}

\subsection{Choice of the local disintegration}
 Let us fix coordinates $(x,z)\in\R\times\R^n$ as in \eqref{eq:loc_frame}. Let $\m$ be a smooth positive measure on $M$ and let $\surf=\{z_n=0\}$ be defined as Section \ref{subsection:assumption}. Under the assumptions \eqref{eq:cond_hypsurf} and \eqref{eq:cond_measure}, we consider the signed distance function $\delta_v$ from $\surf$, as defined in \eqref{eq:dist_surf}, with $v:=z_n$. For ease of notation, we denote $\delta_v$ simply by $\delta$. By triangle inequality, $\delta$ is always $1$-Lipschitz on $M$ with respect to $\di$. However, $\delta$ develops singularities at a characteristic point, indeed it is only H\"older (and not Lipschitz) with respect to the Euclidean distance of the chart, see \cite[Thm.\ 4.2]{MR3741391}. Roughly speaking, such a singularity is related to the fact that the horizontal gradient of $\delta$, which exists almost everywhere (see, \cite[Thm.\ 8]{FHK99}), becomes tangent to $\surf$ as the base point approaches a characteristic point. Our idea is to exploit this behavior to prove that the disintegration associated with $\delta$ does not produce $\cd(K,N)$ densities along the transport rays, for any $K\in\R$, $N\geq 1$. 

Starting from $\delta$, we build a suitable open and bounded set $\Omega$ and we consider the local disintegration of $\m\mres\Omega$ induced by $\delta$, cf. Section \ref{subsection:onedimCD}. Set $B:=B_r((0,\zero))$ for some $r>0$ and define the open and bounded set
\begin{equation}
    \Omega := \bigcup_{q_0\in (\surf\setminus C(\surf))\cap B} \gamma_{q_0},\qquad\text{where}\quad\gamma_{q_0}=G((-f(q_0),f(q_0))\times\{q_0\}),
\end{equation}
where $G$ is the map defined in Lemma \ref{lem:exp_sigma} and $f\in C^\infty(\surf)$ such that $0<f(q_0)<\eps_{q_0}$, for every $q_0\in\surf\setminus C(\surf)$. Note that $G$ is a local diffeomorphism on $\Omega$ and, with this choice of $f$, $\m(\Omega)=\m(\bar\Omega)$. 

Then, $\delta$ is a $1$-Lipschitz function on $\bar \Omega$ inducing a one-dimensional partition in the sense of Definition \ref{def:1d_partition}. Indeed $Q=\{\delta=0\}\cap\bar\Omega=\surf\cap \bar B$ and, for every $q\in Q$, the transport ray $\gamma_q$ of the disintegration coincides with the minimizing geodesics for $\delta$, which exist by completeness. Moreover, $T_\delta=\bar\Omega\setminus C(\surf)$ and therefore, $\m(\bar\Omega\setminus T_\delta)=0$. The quotient map $\mathfrak{Q}\colon \bar\Omega\rightarrow Q$ can be regarded as a projection on the foot of a geodesic\footnote{For any $p\in\bar\Omega$, there exists a unique point ${\mathsf f}(p)\in\surf\cap\bar\Omega$
for which $|\delta(p)| = \di(p, {\mathsf f}(p))$}, thus $\mathfrak{Q}$ is the inverse of the exponential map, namely 
\begin{equation}
    G\left(\delta(p),\mathfrak{Q}(p)\right)=p,\qquad\forall\,p\in\bar\Omega,
\end{equation} 
and $G$ is indeed the ray map associated to the partition. Finally, since $\Omega$ is defined by the smooth function $f$, the measure $\q= \mathfrak{Q}_\# (\m \mres \Omega)$ is smooth on $\surf$.

\subsection{Coordinate expression for the marginals in the disintegration}

Using Lem\-ma \ref{lem:exp_sigma} and, in particular, the diffeomorphism \eqref{eq:diffeo_G}, we can conveniently represent the one-dimensional densities in the disintegration. Indeed, consider a Riemannian point $q_0=(\bar x,\bar z)\in\surf$, with $\bar z_n=0$. In particular, thanks to \eqref{eq:riem_points_surf}, it is enough to assume $\bar x\neq 0$. Then, for every Borel set $C\subset \Omega$, on the one hand we have that
\begin{equation*}
    \int_{\Omega\cap C} \de \m =\int_{\surf}\int_{-f(q)}^{f(q)}\chi_{ G^{-1}(C)} \de (G^*(\m\mres{\Omega})),
\end{equation*}
while, on the other hand, making the disintegration explicit and recalling that $G$ is the ray map, we conclude that
\begin{align}
\int_{\Omega\cap C} \de \m &=\int_{\surf}\m_q(\Omega\cap C)\de \q \\
&= \int_{\surf}\int_{-f(q)}^{f(q)}\chi_C(G(s,q)) \de \m_q\de \q =\int_{\surf}\int_{-f(q)}^{f(q)}\chi_{G^{-1}(C)}(s,q) h_q(s)\de s \de \q.
\end{align}
Thus, having fixed a frame for $T\surf$, say $\{v_1,\ldots,v_n\}$, the density $h_q(s)$ is given by: 
\begin{equation}
\label{eq:density}
\begin{split}
h_q(s) &=h_q(s)ds(\partial_s)\frac{\q(v_1,\ldots,v_n)}{\q(v_1,\ldots,v_n)}=\frac{G^*\m(\partial_s,v_1,\ldots,v_n)}{\q(v_1,\ldots,v_n)}\\
					 &=\frac{\m(G_*\partial_s,G_*v_1,\ldots,G_*v_n)}{\q(v_1,\ldots,v_n)}=\frac{\m(\nabla\delta,G_*v_1,\ldots,G_*v_n)}{\q(v_1,\ldots,v_n)},
\end{split}
\end{equation}
for any $(s,q)\in G^{-1}(\Omega)$, and having used \eqref{eq:grad_delta} in the last equality.

\begin{rmk}
\label{rmk:reg_prop_h}
Notice that, from \eqref{eq:density}, the one-dimensional densities $h_q(s)$ are smooth functions of $(s,q)\in G^{-1}(\Omega)$. Moreover, they do not depend on the choice of coordinates. 
\end{rmk}

We are going to study the second logarithmic derivative of $h_q(s)$, at $s=0$ and as $q\to0$, in order to obtain a contradiction with the differential characterization of Lemma \ref{lem:diff_char}. Firstly, notice that since we are performing derivatives in $s$, we can disregard constant functions in $s$. Secondly, by definition $\m$ is a smooth positive measure, i.e.
\begin{equation}
\label{eq:smooth_density}
\m=m(x,z)dxdz,\qquad \text{with}\quad m\in C^\infty(\R\times\R^n),\quad c\leq m\leq C,
\end{equation}
for some $C,c>0$. Moreover, in \eqref{eq:density}, as a frame for $T\surf$, we can choose the vector fields $\{\partial_x,\partial_{z_1},\ldots,\partial_{z_{n-1}}\}$. In conclusion, we obtain the following expression for the one-dimensional density associated with the disintegration:
\begin{equation}
\label{eq:def_Bq}
    h_q(s)\propto m(G(s,q))dxdz(\nabla\delta,G_*\partial_x,\ldots,G_*\partial_{z_{n-1}})_{|(s,q)}.
\end{equation}
Then, defining the matrix
\begin{equation}
    B_q(s):=(\nabla\delta\mid G_*\partial_x\mid \ldots\mid G_*\partial_{z_{n-1}}),
\end{equation}
where the columns are expressed in coordinates $\{\partial_x,\ldots,\partial_{z_n}\}$, the second logarithm derivative at $s=0$ is given by: 
\begin{equation}
\label{eq:log_derivative}
    \begin{split}
        \left(\log(h_q(s))\right)''_{|s=0} &= \left(\log(m(G(s,q))\right)''_{|s=0}+\left(\log\det( B_q(s)))\right)''_{|s=0} \\
        &= \left(\log(m(G(s,q))\right)''_{|s=0}+\mathrm{tr}\left(-(B_q^{-1}(0)B_q'(0))^2+B_q^{-1}(0)B_q''(0)\right),
    \end{split}
\end{equation}
having used Jacobi formula for the determinant of a smooth curve of invertible matrices:
\begin{equation}
    \det(B(s))'=\det(B(s))\mathrm{tr}\left(B^{-1}(s)B'(s)\right).
\end{equation}

\begin{rmk}
We stress that, for any $q\in Q\setminus C(\surf)$, $h_q(\cdot)$ is defined on an open interval $I_q$ containing $0$. Thus, the derivative in \eqref{eq:log_derivative} makes sense. 
\end{rmk}

 \subsection{Computations for the matrix \texorpdfstring{$B_q(s)$}{Bq(s)}}
Proceeding with hindsight, we analyze the term of \eqref{eq:log_derivative} involving $B_q(s)$, as in general it will be more singular than the other one. We expand in $s$ its columns and we deduce an expression for the coefficients of the expansion, using the \ariem Hamiltonian system.

\subsubsection{An expression for the trace term in \eqref{eq:log_derivative}}

We look for an explicit expression for the matrix $B_q(s)$. We may regard $\nabla\delta\in \R^{n+1}$ and $G\colon\R\times\R^n\rightarrow\R^{n+1}$, therefore expanding in $s$, there exists smooth functions $\f,\h\in C^\infty(\surf\cap \Omega)$ such that,
\begin{equation}
\label{eq:def_fh}
\begin{split}
G(s,q) &=q+\nabla\delta(q)s+\frac{1}{2}\f(q)s^2+o(s^2),\\
\nabla\delta(G(s,q)) &=\nabla\delta(q)+\f(q)s+\frac{1}{2}\h(q)s^2+o(s^2),
\end{split}
\end{equation}
as $s\to0$. The relation between the two expansions comes from \eqref{eq:grad_delta}. Therefore, we obtain the following formulas for $B_q(s)$ and its derivatives at $s=0$: for the zero order term, we have
\begin{equation}
B_q(0)=\left(\nabla\delta\mid \partial_x\mid \partial_{z_1}\mid\ldots\mid \partial_{z_{n-1}}\right)_{|q}=
\begin{pNiceArray}{c|cc}[margin]
& \Block{2-2}{\mathrm{Id}_{n\times n}} \\
\nabla\delta(q)  & & \\
\cline{2-3}
& 0 \hspace{0.25cm}\cdots & 0
\end{pNiceArray}.
\end{equation}
For the first derivative of $B_q(s)$, we differentiate component by component. Notice that we have to take into account the quantities $\partial_{z_i}G(0,q)$, with $i=0,\ldots,n-1$\footnote{Here and below, with a slight abuse of notation, for $i=0$, we set $z_0=x$ and $\partial_{z_0}=\partial_x$.}, therefore we have to differentiate the expansion \eqref{eq:def_fh}, namely: 
\begin{equation}
    \partial_{z_i}G(s,q)=\partial_{z_i}+\partial_{z_i}\nabla\delta(q)s+\frac{1}{2}\partial_{z_i}\f(q)s^2+o(s^2),
\end{equation}
as $s\to0$, for any $i=0,\ldots,n-1$, where the derivatives have to be interpreted component by component. Therefore, we obtain: 
\begin{equation}
B_q'(0)=\left(\f\mid \partial_x\mid \partial_{z_1}\mid\ldots\mid \partial_{z_{n-1}}\right)_{|q}+\sum_{i=0}^{n-1}\left( \nabla\delta\mid \partial_x\mid\ldots\mid \partial_{z_i}\nabla\delta\mid\ldots\mid\partial_{z_{n-1}} \right)_{|q}.
\end{equation}
Analogously, we can deduce the expression for the second-order derivative of $B_q(s)$ at $s=0$:
\begin{align}
B_q''(0) &=\left(\h\mid \partial_x\mid \partial_{z_1}\mid\ldots\mid \partial_{z_{n-1}}\right)_{|q}+\sum_{i=0}^{n-1}\left( \nabla\delta\mid \partial_x\mid\ldots\mid \partial_{z_i}\f\mid\ldots\mid\partial_{z_{n-1}} \right)_{|q}
\\
&\quad+2\sum_{i=0}^{n-2}\sum_{j=i+1}^{n-1}\left( \nabla\delta\mid \partial_x\mid\ldots\mid \partial_{z_i}\nabla\delta\mid\ldots\mid\partial_{z_j}\nabla\delta\mid\ldots\mid\partial_{z_{n-1}} \right)_{|q}
\\
&\quad+2\sum_{i=0}^{n-1}\left( \f\mid \partial_x\mid\ldots\mid \partial_{z_i}\nabla\delta\mid\ldots\mid\partial_{z_{n-1}} \right)_{|q}.
\end{align}
Inserting the above formulas in the trace term in \eqref{eq:log_derivative}, we obtain the desired expression, in terms of the quantities $\nabla\delta$, $\f$ and $\h$.

\subsubsection{Explicit expression for $\nabla\delta$, $\f$ and $\h$}
In order to obtain an explicit expression for $\nabla\delta$, $\f$ and $\h$, we study the Hamiltonian system associated with the \ariem Hamiltonian 
\begin{equation}
H(\lambda)=\frac{1}{2}\sum_{i=0}^n\langle \lambda,X_i\rangle^2,\qquad\forall\,\lambda\in T^*M,
\end{equation}
where $\{X_0,\ldots,X_n\}$ is the local orthonormal frame for the distribution defined in \eqref{eq:loc_frame}. 
In coordinates $(x,z)$, the almost-Riemannian metric $g$ on the Riemannian region is represented by the matrix
\begin{equation}
\begin{pNiceArray}{c|cc}[margin]
1 & &  \\
\hline
& \Block{2-2}{(A^\intercal A)^{-1}}\\
& &\hspace*{1.12cm}
\end{pNiceArray}.
\end{equation}
Therefore the Hamiltonian in canonical coordinates induced by $(x,z)$ is
\begin{equation}
H(p_x,p_z;x,z)=\frac{1}{2}p_x^2+\frac{1}{2}p_z^\intercal A^\intercal A(x,z)p_z,
\end{equation}
where $p_z$ is a shorthand for $(p_{z_1},\ldots,p_{z_n})$. The Hamiltonian system then becomes
\begin{equation}
\label{eq:ham_sys}
\left\{\begin{aligned}
\dot x &=\frac{\partial H}{\partial p_x}=p_x &\qquad& \dot p_x =-\frac{\partial H}{\partial x}=-\frac{1}{2}p_z^\intercal \partial_x\left(A^\intercal A\right)p_z\\
\dot z &=\frac{\partial H}{\partial p_z}=A^\intercal Ap_z &\qquad& \dot p_z=-\frac{\partial H}{\partial z}=-\frac{1}{2}p_z^\intercal \partial_z\left(A^\intercal A\right)p_z
\end{aligned}
\right.
\end{equation}
From Lemma \ref{lem:exp_sigma} (cf. also Remark \ref{rmk:min_geod}) we know that the unique minimizing geodesic for $\delta$ with initial point $q\in\surf\cap\Omega$ has unique (up to a sign) initial covector such that:
\begin{equation}
\label{eq:tr_cond}
\langle\lambda(q),T_q\surf\rangle=0\qquad\text{and}\qquad 2H(\lambda(q))=1. 
\end{equation}
Since $T_q\surf=\spn\{\partial_x,\partial_{z_1},\ldots,\partial_{z_{n-1}}\}$, the first condition in \eqref{eq:tr_cond} implies that $\lambda(q)=p_{z_n}dz_n$. In addition, the second condition in \eqref{eq:tr_cond} forces $\lambda(q)$ to be of the form:
\begin{equation}
\label{eq:covector_coord}
\lambda(q)=\frac{1}{\beta(q)}dz_n,\qquad\text{with}\quad\beta(q)^2=\sum_{k=1}^n a_{kn}(q)^2,
\end{equation}
where we choose $\beta$ to be positive. Thus, denoting by $(x^q(s),z^q(s);p_x^q(s),p_z^q(s))$ the solution to \eqref{eq:ham_sys} with initial datum $(\lambda(q);q)$, the minimizing geodesic for $\delta$ starting at $q$ is given by:
\begin{equation}
\label{eq:min_geod2}
I_q\ni s\mapsto G(s,q)=(x^q(s),z^q(s)),
\end{equation}
where $I_q$ is an open interval containing the origin. By \eqref{eq:grad_delta} and \eqref{eq:min_geod2}, we deduce that 
\begin{equation}
\label{eq:beta_H}
    \nabla\delta(q)=\partial_sG_{|(0,q)}=\left(\dot x^q(0),\dot z^q(0)\right).
\end{equation}
Computing derivatives along $s$ of the equality \eqref{eq:grad_delta} and recalling the definition of $\f,\h$ in \eqref{eq:def_fh}, we analogously obtain higher-order expression in terms of the solution to \eqref{eq:ham_sys}, precisely:
\begin{equation}
\label{eq:fh_H}
\f(q)=\left(\ddot x^q(0),\ddot z^q(0)\right),\qquad\h(q)=\left(\dddot x^q(0),\dddot z^q(0)\right).
\end{equation}
We refer to Appendix \ref{app:ham_sys} for the explicit expression of $\f$ and $\h$. 

\subsection{Contradicting the \texorpdfstring{$\cd$}{CD} condition}
The idea is to exploit the presence of a characteristic point for $\surf$ at the origin to conclude that
\begin{equation}
\label{eq:secdertoinfty}
    \left(\log(h_q(s))\right)''_{|s=0}\xrightarrow{(x,\zero)\to(0,\zero)}+\infty,
\end{equation} 
proving \eqref{eq:disproveCD} for every $K\in \R$, up to taking $x$ sufficiently small. Keeping in mind \eqref{eq:log_derivative}, we anticipate that the term providing the desired pathology will be 
\begin{equation}
    \left(\log\det( B_q(s))\right)''_{|s=0}.
\end{equation} 
Observe that, according to Section \ref{sec:cd}, in order to disprove the $\cd(K,N)$ condition for every $K\in \R$ and $N\in (1,+\infty)$ we need to show that, given any $K\in \R$, it holds 
\begin{equation}
    \left(\log(h_q(s))\right)''_{|s=0}>- K,
\end{equation}
for every $q$ in a $\q$-positive set. However, since the function $(s,q)\to h_q(s)$ is smooth (see Remark \ref{rmk:reg_prop_h}) it is sufficient to prove \eqref{eq:secdertoinfty}. 

Notice that the initial covector for the minimizing geodesic from $\surf$ in \eqref{eq:covector_coord} is singular at $q= (0,\zero)$. More precisely, since $(0,\zero)\in C(\surf)$ and $\surf$ is the level set of $v(x,z)=z_n$, we have that
\begin{equation}
   0=\nabla {z_n}_{|(0,\zero)}=\sum_{i=1}^n X_i(z_n){X_i}_{|(0,\zero)}=\sum_{i,j=1}^n a_{ij}(0,\zero)a_{in}(0,\zero)\partial_{z_j},
\end{equation}
meaning that the function $\beta(q)$ defined in \eqref{eq:covector_coord} vanishes at $q$ \emph{if and only if} $q\in C(\surf)$. In particular, it vanishes at the origin, making the initial covector singular at $q= (0,\zero)$. Moreover, solving the Hamiltonian system, we deduce that
\begin{equation}
\label{eq:nabla_delta_coord}
    \nabla\delta(q)=(\dot x^q(0),\dot z^q(0))=(0,A^\intercal A\beta^{-1}(q)\partial_{z_n})=\left(0,\beta^{-1}(q)\sum_{k=1}^n a_{ki}(q)a_{kn}(q)\right).
\end{equation}

\begin{rmk}
\label{rmk:singularity_delta}
On the one hand, all the components of $\nabla\delta(q)$, but the first and last, are singular at $q= (0,\zero)$, as fast as the initial covector \eqref{eq:covector_coord}. On the other hand, since the last component of $\nabla\delta(q)$ is exactly $\beta(q)$ which tends to $0$ as $q\to(0,\zero)$, formally $\nabla\delta(q)$ becomes tangent to $\surf$ at the characteristic point.   
\end{rmk}

\noindent In particular, we see that $\nabla\delta(q)$ is singular at the origin and the same goes for the functions $\f(q)$, $\h(q)$. Replacing their explicit expressions in \eqref{eq:log_derivative}, we will be able to prove \eqref{eq:secdertoinfty}.

\begin{rmk}
The procedure described in this section for disproving the $\cd$ condition is constructive and the algorithm has been implemented in the software \emph{Mathematica}. The code is available online, see \cite{script}.
\end{rmk}

\section{2-dimensional almost-Riemannian manifolds do not satisfy \texorpdfstring{$\cd$}{CD}}

In this section, we apply our general strategy to show that $2$-dimensional \ariem manifolds do not satisfy any curvature-dimension condition. The reason why we are able to perform explicit computations is related to the better regularity properties of $\delta$, when $\dim M=2$, cf.\ Remark \ref{rmk:singularity_delta}. 

Let $M$ be an almost-Riemannian manifold of dimension $2$, with non-empty singular region $\Z\subset M$. We recall the following local description of a general $2$-dimensional almost-Riemannian manifold which holds \emph{without} any assumption on the structure of the singular set, see \cite[Lem.\ 17]{MR2379474}. 

\begin{lem}
\label{lem:2dim_coord}
Let $M$ be an \ariem manifold. Then, for every point $q_0\in M$, there exists a set of coordinates $\varphi=(x,z)\colon\mathcal{U}\rightarrow M$, centered at $q_0$, such that a local orthonormal frame for the distribution is given by 
\begin{equation}
    X=\partial_x,\qquad Y=f(x,z)\partial_z.
\end{equation}
where $f\colon\mathcal{U}\rightarrow \R$ is a smooth function. Moreover, 
\begin{enumerate}[label=\roman*)]
\item the integral curves of $X$ are normal extremals, as in \eqref{eq:Hamiltoneqs};
\item let $s$ be the step of the structure at $q_0$. If $s = 1$ then $f(0,0) \neq 0$. If $s\geq 2$, we have 
\begin{equation}
\label{eq:vanishing_cond}
    f(0,0)=0,\ \ldots,\  \frac{\partial^{s-2}f}{\partial x^{s-2}}(0,0)=0,\ \frac{\partial^{s-1}f}{\partial x^{s-1}}(0,0)\neq 0.
\end{equation}
\end{enumerate}
\end{lem}

\begin{rmk}
This Lemma improves Lemma \ref{lem:norm_coord} since we can give additional condition on the function $f(x,z)=\det A(x,z)$, using the H\"ormander condition. 
\end{rmk}

In the $2$-dimensional case, assumption \eqref{eq:cond_measure} is always verified, see \cite[Thm.\ 9.14]{ABB-srgeom}. For what concerns assumption \eqref{eq:cond_hypsurf}, we have the following lemma.

\begin{lem}
\label{lem:normal_surf}
Let $M$ be a $2$-dimensional \ariem manifold and let $q_0\in\Z$. Consider the curve in normal coordinates $\surf=\{z=0\}$. Then, up to restricting the chart,  $\surf\cap\Z=C(\surf)=\{(0,0)\}$.
\end{lem}

\begin{proof}
Recall that if $\surf=\{v=0\}$, for $v\in C^\infty$ with never-vanishing differential, then
\begin{equation}
    p\in C(\surf)\quad\Leftrightarrow\quad \nabla v(p)=0, 
\end{equation}
where $\nabla u$ denotes the horizontal gradient of $v$. In particular, in the normal coordinates given by Lemma \ref{lem:2dim_coord}, the singular region is $\Z\cap\mathcal{U}=\{(x,z)\mid f(x,z)=0\}$, thus setting $v(x,z)=z$, 
\begin{equation}
    p=(x,0)\in C(\surf)\quad\Leftrightarrow\quad f(x,0)=0.
\end{equation}
Since $q_0\in\Z$, then $\dim(\D_{q_0})<2$ and the \ariem structure has step $s\geq 2$ at $q_0$. Thus $f(0,0)=0$ and consequently $p=q_0\in C(\surf)$. On the other hand, if $0<|x|<\varepsilon$, $f(x,0)\neq 0$. Indeed, by the vanishing condition \eqref{eq:vanishing_cond} on $f$, we can expand $f(x,0)$ as a Taylor series at $x=0$, obtaining
\begin{equation}
f(x,0)=\frac{\partial^{s-1}f}{\partial x^{s-1}}(0,0) x^{s-1}+o(x^{s-1}),\qquad\text{as }x\to 0.
\end{equation}
where the leading term is not zero. This implies that there exists a smooth function $r\in C^\infty(-\varepsilon,\varepsilon)$, such that $r(x)\neq 0$, for every $x\in (-\varepsilon,\varepsilon)$ and  $f(x,0)=r(x)x^{s-1}$, which never vanishes on $\surf\cap\mathcal{U}\setminus \{q_0\}$, up to restricting the domain of the chart $\mathcal{U}$.
\end{proof}

Now, thanks to Lemmas \ref{lem:2dim_coord} and \ref{lem:normal_surf}, we can follow the general strategy (cf.\ Section \ref{sec:gen_strategy}) to disprove the $\cd$ condition. First of all, notice that the matrix $A=(f(x,z))$ has only one entry, so the Hamiltonian system is greatly simplified. More precisely, the initial covector \eqref{eq:covector_coord} becomes: 
\begin{equation}
    \lambda(x)=\frac{1}{f(x,0)}dz,\qquad\forall\,x\neq 0.
\end{equation}
Thus, as one can check using \eqref{eq:ham_sys}, \eqref{eq:fh_H} and \eqref{eq:nabla_delta_coord} we have
\begin{equation}
\label{eq:2dim_deltafh}
\begin{split}
    \nabla\delta(x) &=(0,f(x,0)),\qquad\f(x) =\left(-\frac{\partial_x f}{f},f\partial_zf\right)_{|(x,0)},\\
    \h(x) &=\left(\star, -\frac{2(\partial_x f)^2}{f}+f(\partial_zf)^2+ f^2\partial^2_zf\right)_{|(x,0)}.
\end{split}
\end{equation}
Here we have omitted the first component of $\h(x)$, since we will not need it. 

Second of all, we can replace the quantities \eqref{eq:2dim_deltafh} in the matrix $B_q(s)$, defined in \eqref{eq:def_Bq}. After a long but routine computation, we obtain the following expression for the logarithmic second derivative of $\det B_q(s)$ at $s=0$, namely
\begin{equation}
    \left(\log \det B_x(s)\right)''(0)=\left(f\partial^2_z f+\frac{(\partial_x f)^2-f\partial_x^2f}{f^2}\right)_{|(x,0)}.
\end{equation}

We are in position to prove the main result of this section. 

\begin{thm}
\label{thm:2dim_cd}
Let $M$ be a complete $2$-dimensional \ariem manifold and let $\m$ be any smooth positive measure on $M$. Then, the metric measure space $(M,\di,\m)$ does not satisfy the $\cd(K,N)$ condition for any $K\in\R$ and $N\in (1,+\infty)$. 
\end{thm}

\begin{proof}
As explained in Section \ref{sec:gen_strategy}, we have to show that the quantity $\left(\log h_x(s)\right)''(0)$ diverges at $+\infty$ as $x\to 0$. Recall that, by the proof of Lemma \ref{lem:normal_surf}, there exists a never-vanishing function $r\in C^\infty(-\varepsilon,\varepsilon)$ such that
\begin{equation}
    f(x,0)=r(x)x^{s-1},\qquad\text{with}\quad f,r\in C^\infty(\mathcal{U}).
\end{equation}
Therefore, using the smoothness of both $f$ and $r$, we deduce that 
\begin{equation}
\label{eq:log_der_matrix}
\begin{split}
     \left(\log \det B_x(s)\right)''(0)&=\frac{(\partial_x f(x,0))^2-f(x,0)\partial_x^2f(x,0)}{f(x,0)^2}+O(1)\\
                                       &=\left(\frac{s-1}{x^2}+\frac{\partial_x r(x)^2-r(x)\partial_x^2r(x)}{r(x)^2}\right)+O(1)\\
                                       &=\frac{s-1}{x^2}+O(1),
\end{split}
\end{equation}

which diverges to $+\infty$ as $x\to 0$, since $(0,0)\in\Z$ and therefore $s>1$. Moreover, let us remark that the singularity is polynomial of order $-2$. We are left to take care of the first term in \eqref{eq:log_derivative}: by a direct computation and using \eqref{eq:smooth_density}, one can check that
\begin{equation}
\label{eq:log_der_densitym}
\begin{split}
   \left| \left(\log m(G(s,x))\right)''(0)\right| &\leq C_0\left(|\partial_sG(0,x)|_e^2+|\partial_s^2G(0,x)|_e\right)\leq C_1 +C_2 \left|\frac{\partial_xf(x,0)}{f(x,0)}\right|\\
   &= C_1 +C_2 \left|\frac{s-1}{x}+\frac{\partial_xr(x)}{r(x)}\right|,
\end{split}
\end{equation}
where $|\cdot|_e$ denotes the Euclidean norm of $\R^2$. Since the singularity in \eqref{eq:log_der_densitym} is polynomial of order $1$, it is negligible compared to the one in \eqref{eq:log_der_matrix}, and we conclude that:
\begin{equation}
    \left(\log \det h_x(s)\right)''(0)\xrightarrow{x\to 0}+\infty,
\end{equation}
disproving the $\cd(K,N)$ condition for any $K\in\R$ and $N\in (0,+\infty)$, as desired.
\end{proof}

\begin{rmk}
A similar argument can be carried out for \emph{generic} $3$-dimensional \ariem manifolds, in the sense of \cite[Def.\ 2]{MR2379474}. Indeed, in this situation we have a convenient description of a local orthonormal frame, and of the matrix $A$, cf. \ \cite[Thm.\ 2]{MR3392620}.
\end{rmk}

\section{Strongly regular almost-Riemannian manifolds do not satisfy \texorpdfstring{$\cd$}{CD}}
\label{sec:strongly_reg}

In this section, we prove that strongly regular \ariem manifolds do not satisfy any curvature-dimension condition. Strongly regular almost-Riemannian manifolds have been studied in \cite{MR3870067, weyl2019}. In this setting, we  can deal with the complexity of the computations thanks to a nice local description of the singularities of the structure. We recall the following definition. 

\begin{defn}
\label{def:strongly_reg}
Let $M$ be a $n$-dimensional \ariem manifold. Assume that the singular set $\Z\subset M$ is an embedded hypersurface without characteristic points. Then, for any $q_0\in\Z$, there exist local coordinates $(x,z)$ centered at $q_0$ such that $\Z=\{x=0\}$ in coordinates, and condition \eqref{eq:loc_frame} is verified, namely a local orthonormal frame for the distribution is given by
\begin{equation}
    X_0=\partial_x,\qquad X_i=\sum_{j=1}^n a_{ij}(x,z)\partial_{z_j}, \quad\forall\,i=1,\ldots,n,
\end{equation}
for some smooth functions $a_{ij}$, so that, denoting by $A=(a_{ij})_{i,j}$, 
\begin{equation}
    \det A(x,z)=0 \qquad\text{if and only if}\qquad x=0.
\end{equation}
We say that $M$ is a \emph{strongly regular} \ariem manifold, if there exists $l\in \mathbb N$ such that 
\begin{equation}
\label{eq:strongly_reg_cond}
    a_{ij}(x,z)=x^l\hat a_{ij}(x,z)\qquad\text{with}\quad\det(\hat a_{ij})(0,z)\neq 0,
\end{equation}
for all $(0,z)$ in the domain of the chart. 
\end{defn}

\begin{rmk}
Although being formulated in coordinates, the notion of a strongly regular \ariem structure on $M$ is intrinsic. In particular, condition \eqref{eq:strongly_reg_cond}, as well as the order $l$, do not depend neither on the choice of $q_0\in\Z$ nor on the coordinates $(x,z)$, see \cite{MR3870067} for further details. 
\end{rmk}

In order to apply our general strategy, we have to ensure that conditions \eqref{eq:cond_hypsurf} and \eqref{eq:cond_measure} are verified. The former is a consequence of the very definition of strongly regular \ariem structure and we pick $\surf$ as in Lemma \ref{lem:nice_hypersurf} so that also the latter condition is satisfied. We proceed by computing the second logarithmic derivative of the one-dimensional densities,
\begin{equation}
\label{eq:expr_ddot_B}
\begin{split}
\left(\log \det B_q(s)\right)''_{|s=0}&=\frac{1}{\beta}\Bigg[\h_n+\beta\partial_x\f_0+\sum_{i=1}^{n-1}\left(\beta\partial_{z_i}\f_i-\beta_i\partial_{z_i}\f_n\right)-2\f_0\partial_x\beta\\
&\quad+2\sum_{i=1}^{n-1}\left(\f_n\partial_{z_i}\beta_i-\f_i\partial_{z_i}\beta\right)+2\sum_{0< i<j< n}\det\begin{pNiceArray}{ccc}[margin]
\partial_{z_i}\beta_i & \partial_{z_j}\beta_i & \beta_i\\
\partial_{z_i}\beta_j & \partial_{z_j}\beta_j & \beta_j\\
\partial_{z_i}\beta   & \partial_{z_j}\beta   & \beta
\end{pNiceArray} \Bigg]\\
&\quad-\frac{1}{\beta^2}\left(\f_n+\sum_{i=1}^{n-1}\left(\partial_{z_i}\beta_i\beta-\beta_i\partial_{z_i}\beta\right)\right)^2,
\end{split}
\end{equation}
where $\beta_i$, $\f_i$, $\h_i$ denote the components of $\nabla\delta$, $\f$ and $\h$ respectively. This computation follows from the trace term in \eqref{eq:log_derivative}, using the property that the first component of $\nabla\delta$ is identically zero, cf.\ \eqref{eq:nabla_delta_coord}.

\begin{thm}
\label{thm:strongly_reg_cd}
Let $M$ be a complete strongly regular \ariem manifold and let $\m$ be any smooth positive measure on $M$. Then, the metric measure space $(M,\di,\m)$ does not satisfy the $\cd(K,N)$ condition for any $K\in\R$ and $N\in (1,+\infty)$. 
\end{thm}

\begin{proof}
As in the proof of Theorem \ref{thm:2dim_cd}, we have to show that the quantity $\left(\log h_q(s)\right)''(0)$ diverges at $+\infty$ as $q\to 0$ along $\surf$. To do that, the idea is to highlight the most singular terms in $x$, namely those where a derivative in $x$ appears. Let us discuss the order in $x$ of the quantities in \eqref{eq:expr_ddot_B}, using as well formulas from Appendix \ref{app:ham_sys}. Firstly, since $M$ is strongly regular, \eqref{eq:strongly_reg_cond} holds and we have 
\begin{equation}
\label{eq:beta_strongly_reg}
    \beta(x,z)^2=\sum_{i=1}^n a_{kn}^2(x,z)=x^{2l}\sum_{i=1}^n \hat a_{kn}^2(x,z)=x^{2l}\hat\beta(x,z)^2,
\end{equation}
with $\hat\beta(0,z)\neq 0$. Thus, $\beta$ has order $l$ in $x$. Similarly, the components $\beta_i$ of $\nabla\delta$ are given by \eqref{eq:beta_i},
\begin{equation}
    \beta_i(x,z)=\frac{\alpha_i(x,z)}{\beta(x,z)}=\frac{1}{\beta(x,z)}\sum_{k=1}^na_{ki}(x,z)a_{kn}(x,z)=\frac{ x^l}{\hat\beta(x,z)}\sum_{k=1}^n\hat a_{ki}(x,z)\hat a_{kn}(x,z).
\end{equation}
Therefore, also $\beta_i$'s have order $l$ in $x$. A crucial remark before moving forward is that, thanks to the strongly regular assumption on $M$, computing derivatives along $z$-directions does not change the order in $x$ of the quantities. Thus, for example, 
\begin{equation}
    \ord_x \partial_{z_j}\beta_i(x,z)=l,\qquad \forall\,i,j=1,\ldots,n.
\end{equation}
Reasoning in this way, for the functions $\f_i$ defined in \eqref{eq:f_i}, we have: 
\begin{equation}
\ord_x \f_i(x,z)=2l,\qquad\forall\,i=1,\ldots,n,
\end{equation}
and the same is true for any derivative in $z$-directions. For what concerns $\f_0$, recall that
\begin{equation}
\label{eq:ord_f_0}
    f_0(x,z)=-\frac{\partial_x\beta(x,z)}{\beta(x,z)}\qquad\Rightarrow\qquad\ord_x \f_0(x,z)=-1.
\end{equation}
From \eqref{eq:ord_f_0}, it is clear that derivatives in the $x$-direction encode all the possible singularities of second logarithmic derivatives of $h_q(s)$. Finally, using \eqref{eq:h_n}, we see that 
\begin{equation}
    \ord_x\h_n(x,z)=l-2,\qquad\text{and}\qquad\h_n(x,z)=-\frac{2(\partial_x\beta(x,z))^2}{\beta(x,z)}+O(x^{3l}).
\end{equation}
Finally, we can evaluate the order in $x$ of the functions in \eqref{eq:expr_ddot_B}: the lowest order is $-2$ coming from the terms $\h_n\beta^{-1}$, $\partial_x\f_0$ and $\beta^{-1}\f_0\partial_x\beta$. Thus, denoting by $z'=(z_1,\ldots,z_{n-1},0)$, we obtain:
\begin{equation}
\begin{split}
   (\log\det B_q(s))''_{|s=0} &=\left(\frac{\h_n(x,z')-2\f_0\partial_x\beta(x,z')}{\beta(x,z')}+\partial_x\f_0(x,z')\right)+O(1)\\
   &=\left(\frac{-\partial^2_x\beta(x,z')\beta(x,z')+(\partial_x\beta(x,z'))^2}{\beta^2(x,z')}\right)+O(1).
\end{split}
\end{equation}

Now using \eqref{eq:beta_strongly_reg}, we can reason as in the $2$-dimensional case, cf.\ \eqref{eq:log_der_matrix}, to conclude that
\begin{equation}
     \left(\log \det B_q(s)\right)''(0)\xrightarrow{q\to (0,\zero)}+\infty.
\end{equation} 
Once again, also in this situation, the singularity in $x$ of the quantity $ \left(\log \det B_q(s)\right)''(0)$ is polynomial of order $-2$. Finally, using the same argument used in \eqref{eq:log_der_densitym} for the $2$-dimensional case, we can show that the density of the measure $\m$ produces a polynomial singularity of order $-1$, which is negligible as $q\to(0,\zero)$. Finally, we obtain
\begin{equation}
    \left(\log \det h_q(s)\right)''(0)\xrightarrow{q\to (0,\zero)}+\infty,
\end{equation}
disproving the $\cd(K,N)$ condition for any $K\in\R$ and $N\in (1,+\infty)$, as desired. 
\end{proof}

\begin{rmk}
We stress once again that, thanks to the strongly regular assumption on $M$, the order of the structure (and thus the order of $\beta$) controls the orders in $x$, not only of the functions $\beta_i$, $\f_i$ and $\h_n$, but also of their derivatives in the $z$-directions. Below, we provide an example of regular (but not strongly regular) structure where the orders of the derivatives are not controlled by the order of $\beta$. Nevertheless our strategy to disprove the $\cd$ condition works. 

In full generality, it is possible to prove that $ \left(\log \det h_x(s)\right)''(0)$ actually diverges, however there is no criterion of determining the \emph{sign} of the leading order, without requiring some additional regularity on the structure. On the other hand, as characteristic points encode the truly \sr behavior of \ariem manifolds, we believe that our strategy should always be effective.
\end{rmk}

\begin{example}
Let $M=\R^4$ and in coordinates $(x,z_1,z_2,z_3)$ consider the \ariem structure defined by the global vector fields
\begin{equation}
    X_0=\partial_x,\quad X_1=\partial_{z_1}-\frac{z_2}{2}\partial_{z_3},\quad X_2=\partial_{z_2}+\frac{z_1}{2}\partial_{z_3},\quad X_3=x\partial_{z_3}.
\end{equation}
The singular region is given by $\Z=\{x=0\}$ and is an embedded hypersurface without characteristic points. Notice that $M$ is regular, see \cite[Def.\ 7.10]{MR3870067} for the precise definition, but not strongly regular, thus we can not apply Theorem \ref{thm:strongly_reg_cd}. Nevertheless, if we consider $\surf=\{z_3=0\}$, assumptions \eqref{eq:cond_hypsurf} and \eqref{eq:cond_measure} are verified, therefore, we can apply our general strategy. Setting $\m=\Leb^4$, an explicit computation leads to 
\begin{equation}
\label{eq:monito}
     \left(\log \det h_q(s)\right)''(0)=\frac{8x^2-4(z_1^2+z_2^2)}{(4x^2+z_1^2+z_2^2)^2},
\end{equation}
which diverges at $+\infty$ along the curve $(x,0,0,0)$ as $x\to 0$, disproving the $\cd(K,N)$ condition for any $K\in\R$ and $N\geq 1$. A few remarks are in order: first of all, the function $\beta(x,z_1,z_2)=4x^2+z_1^2+z_2^2$ has order $2$ in $x$ but this is not true for its derivatives in the $z$-directions. Second of all, the numerator of \eqref{eq:monito} does not have a sign, highlighting the difficulties of the general case of determining the behavior of the leading term.
\end{example}

\appendix

\section{Explicit expression for \texorpdfstring{$\nabla\delta$, $\f$ and $\h$}{grad(delta), f, h}}
\label{app:ham_sys}
In order to obtain an explicit expression for $\nabla\delta$, $\f$ and $\h$, we study the Hamiltonian system associated with the \sr Hamiltonian. Recall that in canonical coordinates induced by $(x,z)$, given by \eqref{eq:loc_frame}, the Hamiltonian is
\begin{equation}
H(p_x,p_z;x,z)=\frac{1}{2}p_x^2+\frac{1}{2}p_z^\intercal A^\intercal A(x,z)p_z,
\end{equation}
where $p_z$ is a shorthand for $(p_{z_1},\ldots,p_{z_n})$. For the hypersurface $\surf\subset M$, given by \eqref{eq:cond_hypsurf}, from Lemma \ref{lem:exp_sigma}, we know that the unique minimizing geodesic for $\delta$ with initial point $q\in\surf\setminus C(\surf)$ has unique (up to a sign) initial covector: 
\begin{equation}
\lambda(q)=\frac{1}{\beta(q)}dz_n,\qquad\text{with}\quad\beta(q)^2=\sum_{k=1}^n a_{kn}(q)^2
\end{equation}
Thus, if $(x(s),z(s),p_x(s),p_z(s))$\footnote{We drop the superscript $q$ only in this section to ease the notation.} is the solution to \eqref{eq:ham_sys} with initial datum $(\lambda(q);q)$, we deduce that 
\begin{equation}
\label{eq:beta_i}
    \beta_0(q)=\dot x(0)=0,\qquad\beta_i(q)=\dot z_i(0)=\frac{1}{\beta(q)}\sum_{k=1}^n a_{ki}(q)a_{kn}(q)=\frac{\alpha_i(q)}{\beta(q)},\quad\forall\,i=1,\ldots,n,
\end{equation}
having denoted $\nabla\delta(q)=\left(\beta_0(q),\ldots,\beta_n(q)\right)$. Moreover, notice that by definition $\beta_n(q)=\beta(q)$. In an analogous way, we can compute $\f=\left(\f_0,\ldots,\f_n\right)$:
\begin{equation}
\label{eq:f_i}
\begin{split}
\f_0(q)&=\ddot x(0)=-\frac{\partial_x \beta(q)}{\beta(q)}\\
\f_i(q)&=\ddot z_i(0)=\frac{1}{\beta^2(q)}\left[\sum_{l=1}^n\partial_{z_l}\alpha_i(q)\alpha_l(q)-\frac{1}{2}\sum_{j,k=1}^n a_{ki}(q)a_{kj}(q)\partial_{z_j}\beta^2(q)\right].
\end{split}
\end{equation}
Finally, taking the third-order derivatives in $s$ of the solution to \eqref{eq:ham_sys}, we obtain $\h$. Notice, however, that we only need the $n$-th component of $\h$ in \eqref{eq:expr_ddot_B}, thus: 
\begin{multline}
\label{eq:h_n}
\h_n =\dddot{z}_n(0) =\frac{1}{\beta^3}\Bigg[-\frac{(\partial_x \beta^2)^2}{2}+\sum_{j,r,l=1}^n\alpha_l\alpha_r\partial^2_{z_lz_r}(\beta^2)+\sum_{j,l=1}^n\beta^2\partial_{z_l}(\beta^2)\f_l\\					 -\sum_{j,l=1}^n\alpha_l\partial_{z_l}\alpha_j\partial_{z_j}(\beta^2)-\frac{1}{2}\sum_{j,l=1}^n \alpha_j\left( \alpha_l\partial_{z_jz_l}(\beta^2)- \partial_{z_j}\alpha_l\partial_{z_l}(\beta^2)\right)\Bigg].		
\end{multline}

\bibliographystyle{alphaabbr}
\bibliography{biblio-cd-asr}

\begin{thebibliography}{BCGM15}

\bibitem[ABB20]{ABB-srgeom}
A.~Agrachev, D.~Barilari, and U.~Boscain.
\newblock {\em A comprehensive introduction to sub-{R}iemannian geometry},
  volume 181 of {\em Cambridge Studies in Advanced Mathematics}.
\newblock Cambridge University Press, Cambridge, 2020.
\newblock From the Hamiltonian viewpoint, With an appendix by Igor Zelenko.

\bibitem[ABR18]{MR3852258}
A.~Agrachev, D.~Barilari, and L.~Rizzi.
\newblock Curvature: a variational approach.
\newblock {\em Mem. Amer. Math. Soc.}, 256(1225):v+142, 2018.

\bibitem[ABS08]{MR2379474}
A.~Agrachev, U.~Boscain, and M.~Sigalotti.
\newblock A {G}auss-{B}onnet-like formula on two-dimensional
  almost-{R}iemannian manifolds.
\newblock {\em Discrete Contin. Dyn. Syst.}, 20(4):801--822, 2008.

\bibitem[ACS18]{MR3741391}
P.~Albano, P.~Cannarsa, and T.~Scarinci.
\newblock Regularity results for the minimum time function with {H}\"{o}rmander
  vector fields.
\newblock {\em J. Differential Equations}, 264(5):3312--3335, 2018.

\bibitem[AS04]{AS-GeometricControl}
A.~Agrachev and Y.~L. Sachkov.
\newblock {\em Control theory from the geometric viewpoint}, volume~87 of {\em
  Encyclopaedia of Mathematical Sciences}.
\newblock Springer-Verlag, Berlin, 2004.
\newblock Control Theory and Optimization, II.

\bibitem[BCGM15]{MR3392620}
U.~Boscain, G.~Charlot, M.~Gaye, and P.~Mason.
\newblock Local properties of almost-{R}iemannian structures in dimension 3.
\newblock {\em Discrete Contin. Dyn. Syst.}, 35(9):4115--4147, 2015.

\bibitem[BKS19]{MR4019096}
Z.~M. Balogh, A.~Krist\'{a}ly, and K.~Sipos.
\newblock Jacobian determinant inequality on corank 1 {C}arnot groups with
  applications.
\newblock {\em J. Funct. Anal.}, 277(12):108293, 36, 2019.

\bibitem[BR18]{MR3848070}
D.~Barilari and L.~Rizzi.
\newblock Sharp measure contraction property for generalized {H}-type {C}arnot
  groups.
\newblock {\em Commun. Contemp. Math.}, 20(6):1750081, 24, 2018.

\bibitem[BR19]{MR3935035}
D.~Barilari and L.~Rizzi.
\newblock Sub-{R}iemannian interpolation inequalities.
\newblock {\em Invent. Math.}, 215(3):977--1038, 2019.

\bibitem[BR20]{MR4245620}
Z.~Badreddine and L.~Rifford.
\newblock Measure contraction properties for two-step analytic sub-{R}iemannian
  structures and {L}ipschitz {C}arnot groups.
\newblock {\em Ann. Inst. Fourier (Grenoble)}, 70(6):2303--2330, 2020.

\bibitem[Cav14]{cavalletti2013}
F.~Cavalletti.
\newblock Monge problem in metric measure spaces with {R}iemannian
  curvature-dimension condition.
\newblock {\em Nonlinear Anal.}, 99:136--151, 2014.

\bibitem[CM17]{MR3648975}
F.~Cavalletti and A.~Mondino.
\newblock Sharp and rigid isoperimetric inequalities in metric-measure spaces
  with lower {R}icci curvature bounds.
\newblock {\em Invent. Math.}, 208(3):803--849, 2017.

\bibitem[CM20]{MR4175820}
F.~Cavalletti and A.~Mondino.
\newblock New formulas for the {L}aplacian of distance functions and
  applications.
\newblock {\em Anal. PDE}, 13(7):2091--2147, 2020.

\bibitem[CM21]{MR4309491}
F.~Cavalletti and E.~Milman.
\newblock The globalization theorem for the curvature-dimension condition.
\newblock {\em Invent. Math.}, 226(1):1--137, 2021.

\bibitem[CPR19]{weyl2019}
Y.~Chitour, D.~Prandi, and L.~Rizzi.
\newblock Weyl's law for singular {R}iemannian manifolds, 2019.
\newblock arXiv prepint 1903.05639.

\bibitem[FHK99]{FHK99}
B.~Franchi, P.~Haj{\l}asz, and P.~Koskela.
\newblock Definitions of {S}obolev classes on metric spaces.
\newblock {\em Ann. Inst. Fourier (Grenoble)}, 49(6):1903--1924, 1999.

\bibitem[FPR20]{FPR-sing-lapl}
V.~Franceschi, D.~Prandi, and L.~Rizzi.
\newblock On the essential self-adjointness of singular sub-{L}aplacians.
\newblock {\em Potential Anal.}, 53(1):89--112, 2020.

\bibitem[Jui10]{MR2648260}
N.~Juillet.
\newblock On a method to disprove generalized {B}runn-{M}inkowski inequalities.
\newblock In {\em Probabilistic approach to geometry}, volume~57 of {\em Adv.
  Stud. Pure Math.}, pages 189--198. Math. Soc. Japan, Tokyo, 2010.

\bibitem[Jui21]{MR4201410}
N.~Juillet.
\newblock Sub-{R}iemannian structures do not satisfy {R}iemannian
  {B}runn-{M}inkowski inequalities.
\newblock {\em Rev. Mat. Iberoam.}, 37(1):177--188, 2021.

\bibitem[LV09]{MR2480619}
J.~Lott and C.~Villani.
\newblock Ricci curvature for metric-measure spaces via optimal transport.
\newblock {\em Ann. of Math. (2)}, 169(3):903--991, 2009.

\bibitem[Mil21]{MR4373164}
E.~Milman.
\newblock The quasi curvature-dimension condition with applications to
  sub-{R}iemannian manifolds.
\newblock {\em Comm. Pure Appl. Math.}, 74(12):2628--2674, 2021.

\bibitem[MR20]{MR4153911}
T.~Mietton and L.~Rizzi.
\newblock Branching geodesics in sub-{R}iemannian geometry.
\newblock {\em Geom. Funct. Anal.}, 30(4):1139--1151, 2020.

\bibitem[MR22]{script}
M.~Magnabosco and T.~Rossi.
\newblock {A}n algorithmic procedure to disprove the {$\cd$} condition for
  almost-{R}iemannian manifolds, 2022.
\newblock
  \href{https://github.com/TRenghia/CD-on-AR-manifolds}{https://github.com/TRenghia/CD-on-AR-manifolds}.

\bibitem[Oht07]{MR2341840}
S.-i. Ohta.
\newblock On the measure contraction property of metric measure spaces.
\newblock {\em Comment. Math. Helv.}, 82(4):805--828, 2007.

\bibitem[PRS18]{MR3870067}
D.~Prandi, L.~Rizzi, and M.~Seri.
\newblock Quantum confinement on non-complete {R}iemannian manifolds.
\newblock {\em J. Spectr. Theory}, 8(4):1221--1280, 2018.

\bibitem[Riz16]{MR3502622}
L.~Rizzi.
\newblock Measure contraction properties of {C}arnot groups.
\newblock {\em Calc. Var. Partial Differential Equations}, 55(3):Art. 60, 20,
  2016.

\bibitem[RS23]{rizzistefani}
L.~Rizzi and G.~Stefani.
\newblock Failure of curvature-dimension conditions on sub-{R}iemannian
  manifolds via tangent isometries, 2023.

\bibitem[Stu06a]{MR2237206}
K.-T. Sturm.
\newblock On the geometry of metric measure spaces. {I}.
\newblock {\em Acta Math.}, 196(1):65--131, 2006.

\bibitem[Stu06b]{MR2237207}
K.-T. Sturm.
\newblock On the geometry of metric measure spaces. {II}.
\newblock {\em Acta Math.}, 196(1):133--177, 2006.

\end{thebibliography}

\end{document}